\documentclass[12pt]{amsart}
\usepackage{preamble}

\begin{document}

\title[Vincular pattern avoidance on cyclic permutations]{Vincular pattern avoidance on cyclic permutations}
\author[Rupert Li]{Rupert Li}
\address{Massachusetts Institute of Technology, 77 Massachusetts Avenue, Cambridge, MA 02139, USA}
\email{rupertli@mit.edu}

\date{\today}
\begin{abstract}
    Pattern avoidance for permutations has been extensively studied, and has been generalized to vincular patterns, where certain elements can be required to be adjacent.
    In addition, cyclic permutations, i.e., permutations written in a circle rather than a line, have been frequently studied, including in the context of pattern avoidance.
    We investigate vincular pattern avoidance on cyclic permutations.
    In particular, we enumerate many avoidance classes of sets of vincular patterns of length 3, including a complete enumeration for all single patterns of length 3.
    Further, several of the avoidance classes corresponding to a single vincular pattern of length 4 are enumerated by the Catalan numbers.
    We then study more generally whether sets of vincular patterns of an arbitrary length $k$ can be avoided for arbitrarily long cyclic permutations, in particular investigating the boundary cases of minimal unavoidable sets and maximal avoidable sets.
\end{abstract}

\maketitle

\section{Introduction}\label{section: introduction}
Pattern containment and avoidance for permutations is a well-established branch of enumerative combinatorics; see Kitaev \cite{kitaev2011patterns} for a further introduction.
The study of pattern avoidance was generalized to vincular patterns in 2000 by Babson and Steingr\'imsson \cite{babson2000generalized}, where vincular patterns can additionally require some elements to be adjacent when considering whether a permutation contains the pattern; see Steingr\'imsson  \cite{steingrimsson2010generalized} for a survey of the study of vincular patterns, which he refers to as generalized patterns.

A frequently studied variant of permutations is cyclic permutations, where the permutation is written in a circle rather than a line.
Cyclic permutations are frequently encountered outside of the context of pattern avoidance; for a recent example, Kim and Williams \cite{kim2021schubert} used cyclic permutations to study the inhomogeneous totally asymmetric simple exclusion process on a ring.
In 2002, Callan \cite{callan2002pattern} initiated the study of pattern avoidance in cyclic permutations, enumerating the avoidance classes for all patterns of length 4.
In 2021, Domagalski, Liang, Minnich, Sagan, Schmidt, and Sietsema \cite{domagalski2021cyclic} extended Callan's work, enumerating the avoidance classes for all sets of patterns of length 4, where a permutation avoids a set of patterns if it avoids each pattern in the set.
Menon and Singh \cite{menon2021pattern} extend these results to patterns of higher length, principally investigating pairs of patterns, one of length 4 and the other of length $k$.

Domagalski et al.\ additionally initiated the study of vincular pattern avoidance in cyclic permutations, showing that the Catalan numbers appear as the sizes of a particular avoidance class of cyclic permutations.

In this paper, we provide a more thorough, foundational investigation of vincular pattern avoidance on cyclic permutations.
In \cref{section: vincular 3}, we enumerate the avoidance classes of all vincular cyclic patterns of length 3, the first nontrivial length to enumerate.
We extend this analysis in \cref{section: multiple vincular 3}, where we enumerate the avoidance classes of all sets of at least three vincular cyclic patterns of length 3, as well as some of the doubleton sets of patterns.
In particular, in \cref{subsection: multiple vincular 3 doubleton} we find one of the doubleton sets is equinumerous to the set of up-down permutations, and in \cref{subsection: multiple vincular 3 three+} we find the unique nonzero Wilf class of tripleton sets of patterns of length 3 has enumeration equivalent to finding the cardinality of the set of total extensions of a certain partial cyclic order, a circular analog of a poset, for which a recurrence is known.
In \cref{section: vincular 4}, we enumerate six of the eight trivial Wilf equivalence classes of vincular cyclic patterns of length 4 with a single vinculum, demonstrating that there are five Wilf equivalence classes;
combining this with the result by Domagalski, Liang, Minnich, Sagan, Schmidt, and Sietsema \cite{domagalski2021cyclic}, this leaves only one of the eight trivial Wilf equivalence classes unresolved.
Notably, we show that the cyclic permutations of a given length avoiding a member of two more trivial Wilf classes are enumerated by the Catalan numbers.

In \cref{section: unavoidable sets of totally vincular patterns,section: max avoidable sets of totally vincular patterns}, we investigate vincular pattern avoidance on cyclic permutations for patterns of general lengths.
In particular, we consider whether a given set of totally vincular patterns, i.e., patterns where the entire subsequence must be consecutive, is unavoidable, meaning no sufficiently long cyclic permutation can avoid this set.
In \cref{section: unavoidable sets of totally vincular patterns}, we consider the boundary cases of this property, namely the minimal unavoidable sets, which are sets of totally vincular patterns that are unavoidable but for which any proper subset is avoidable.
We demonstrate that one of the most natural families of unavoidable sets, namely the sets of patterns with a 1 at position $i$ for a fixed $i$, is minimal.
And in \cref{section: max avoidable sets of totally vincular patterns}, we consider the dual question, maximal avoidable sets, which are sets of totally vincular patterns that are avoidable but adding any additional pattern makes the set unavoidable.
We determine the maximum cardinality of any avoidable set.

Finally, we conclude in \cref{section: conclusion} with some remarks on areas for further research.

\section{Preliminaries}\label{section: preliminaries}
Let $a,b\in \Z$.
When we write the interval $[a,b]$, we will only be referring to the integers in that interval.
As is standard, we will use $[a]$ to denote the set $\{1,2,\dots,a\}$.

Let $S_n$ denote the set of permutations on $[n]$ for a positive integer $n$.
Any permutation $\pi \in S_n$ is said to have \emph{length} $n$, denoted by $|\pi|=n$.
A permutation $\pi$ will often be written in the one-line notation $\pi=\pi_1\cdots\pi_n$, where commas may be optionally inserted such as $\pi=\pi_1,\dots,\pi_n$ for sake of readability.

In particular, two of the simplest permutations of length $n$ are the increasing and decreasing permutations, which will appear throughout the paper, denoted
\[ \iota_n = 12\cdots n\]
and
\[ \delta_n = n\cdots21,\]
respectively.

Two sequences of distinct integers $\pi = \pi_1\cdots\pi_k$ and $\sigma = \sigma_1\cdots\sigma_k$ of the same length are \emph{order isomorphic}, denoted $\pi \cong \sigma$, if $\pi_i < \pi_j$ if and only if $\sigma_i < \sigma_j$ for $1 \leq i,j \leq n$.
The \emph{reduction} of a sequence of distinct integers $\pi = \pi_1\cdots\pi_n$ is the unique permutation $\sigma\in S_n$ such that $\pi \cong \sigma$; in other words, the values of the elements of $\pi$ are mapped to $[n]$ while preserving their relative order.

We now define pattern avoidance on permutations.
If $\sigma \in S_n$ and $\pi \in S_k$ for $k \leq n$, then $\sigma$ \emph{contains} $\pi$ as a pattern if there is a subsequence $\sigma'$ of $\sigma$ with $|\sigma'|=k$ such that $\sigma' \cong \pi$.
If no such subsequence exists, then $\sigma$ is said to \emph{avoid} $\pi$.
The \emph{avoidance class} of $\pi$ is
\[ \Av_n(\pi)=\{\sigma\in S_n \mid \sigma \text{ avoids } \pi\}.\]
We extend this definition to avoidance classes of sets of permutations $\Pi$ by defining
\[ \Av_n(\Pi) = \bigcap_{\pi \in \Pi} \Av_n(\pi).\]

The \emph{reverse} of a permutation $\pi=\pi_1\cdots\pi_n$ is $\pi^r = \pi_n \cdots \pi_1$.
We define the \emph{plot} of a permutation $\pi$ to be the sequence of points $(i,\pi_i)$ in the Cartesian plane.
Reversal then corresponds to reflecting the plot of a permutation across a vertical axis.
Similarly, reflection across a horizontal axis corresponds to the \emph{complement} of $\pi$, given by
\[ \pi^c = n+1-\pi_1, n+1-\pi_2, \dots, n+1-\pi_n. \]
Combining these two operations gives the \emph{reverse complement}
\[ \pi^{rc} = n+1-\pi_n, \dots, n+1-\pi_1,\]
which corresponds to rotation by 180 degrees.
We can apply any of these three operations to sets of permutations $\Pi$ by applying them to each element of $\Pi$.
In other words, we have
\begin{align*}
    \Pi^r &= \{ \pi^r \mid \pi \in \Pi\} \\
    \Pi^c &= \{ \pi^c \mid \pi \in \Pi\} \\
    \Pi^{rc} &= \{ \pi^{rc} \mid \pi \in \Pi\}.
\end{align*}

We say that two patterns $\pi$ and $\pi'$ are \emph{Wilf equivalent}, written $\pi \equiv \pi'$, if for all $n \geq 1$, we have $\left|\Av_n(\pi)\right|=\left|\Av_n(\pi')\right|$.
We extend this definition naturally to sets of patterns, denoted $\Pi \equiv \Pi'$.
It is easy to see that $\pi \equiv \pi^r \equiv \pi^c \equiv \pi^{rc}$ for any pattern $\pi$, so these are called \emph{trivial Wilf equivalences}.
This naturally generalizes to trivial Wilf equivalences for sets of patterns: $\Pi \equiv \Pi^r \equiv \Pi^c \equiv \Pi^{rc}$.
These relations form \emph{trivial Wilf equivalence classes}.

For $\sigma=\sigma_1\cdots\sigma_n\in S_n$, let a \emph{rotation} of $\sigma$ be any permutation $\tau\in S_n$ of the form
\[ \tau = \sigma_k\sigma_{k+1}\cdots\sigma_n\sigma_1\cdots\sigma_{k-1}\]
for some $k \in [n]$.
We define the \emph{cyclic permutation} corresponding to $\sigma \in S_n$ to be the set of all rotations of $\sigma$, denoted
\[ [\sigma] =\{\sigma_1\cdots\sigma_n,\sigma_2\cdots\sigma_n\sigma_1,\dots,\sigma_n\sigma_1\cdots\sigma_{n-1} \}. \]
For example,
\[ [123] = \{123, 231, 312\} = [231] = [312].\]
We use square brackets to denote the cyclic analog of objects defined in the linear case, and using this notation, we let $[S_n]$ denote the set of cyclic permutations of length $n$.
Notice that $\left|[S_n]\right| = (n-1)!$.
To avoid confusion, we may refer to permutations from $S_n$ as \emph{linear} permutations, as opposed to cyclic permutations.
The \emph{length} of a cyclic permutation $[\sigma]$ is simply the length of $\sigma$, so any permutation $[\sigma]\in[S_n]$ has length $n$, even though one may view a cyclic permutation as being arranged in a circle and thus lacking endpoints.

We now define pattern avoidance on cyclic permutations, the natural analog to the definition of pattern avoidance in the linear case.
If $[\sigma]\in [S_n]$ and $[\pi]\in [S_k]$ for $k \leq n$, then $[\sigma]$ \emph{contains} $[\pi]$ as a pattern if some element $\sigma' \in [\sigma]$, i.e., some rotation $\sigma'$ of $\sigma$, contains $\pi$ as a pattern, using the linear definition of pattern avoidance.
Otherwise, $[\sigma]$ is said to \emph{avoid} $[\pi]$.
Intuitively, a cyclic permutation can be written in a circle rather than a line as in the linear case, and a cyclic permutation contains a pattern if some subsequence, going once around the circle, is order isomorphic to the pattern.
The formal definition of a cyclic permutation being the set of rotations of a linear rotation enforces that one cannot loop around the circle multiple times; if this were allowed, then any cyclic permutation of length $n$ would trivially contain any pattern of length $k$ for $k \leq n$.

The \emph{avoidance class} of $[\pi]$ is similarly defined to be
\[ \Av_n[\pi] = \{[\sigma] \in [S_n] \mid [\sigma] \text{ avoids } [\pi] \}. \]
As before, we extend this definition to avoidance classes of sets of cyclic permutations $[\Pi]$ by defining
\[ \Av_n[\Pi] = \bigcap_{[\pi]\in[\Pi]} \Av_n[\pi]. \]
For simplicity, when working with an explicit set of patterns we may omit the curly brackets for the set; for example,
\[ \Av_n[1234,1243] = \Av_n[\{1234,1243\}].\]

Wilf equivalences for cyclic permutations and sets of cyclic permutations are defined analogously to the linear case.
In particular, we still have the trivial Wilf equivalences: for all $[\pi]$ and $[\Pi]$, we have
\begin{align*}
    [\pi] \equiv [\pi^r] &\equiv [\pi^c] \equiv [\pi^{rc}] \\
    [\Pi] \equiv [\Pi^r] &\equiv [\Pi^c] \equiv [\Pi^{rc}].
\end{align*}

Lastly, we introduce vincular patterns and vincular pattern avoidance.
We consider $\pi$ as a vincular pattern if, when determining which permutations $\sigma$ avoid $\pi$, we only consider subsequences $\sigma'$ of $\sigma$ where certain adjacent elements of $\pi$ are also adjacent in $\sigma'$ when $\sigma'$ is embedded within $\sigma$.
Such adjacent elements are overlined in $\pi$, and each adjacency requirement, i.e., each adjacent pair of elements that are overlined together, is referred to as a vinculum.
When two vincula are themselves adjacent, the overlines are combined into one longer overline.
For example, $\sigma=34251$ contains two subsequences order isomorphic to $\pi = 213$, namely $325$ and $425$, but only $425$ is a copy of $\pi'=\overline{21}3$, because the 4 and the 2 are adjacent.
In fact, $425$ is a copy of $\pi''=\overline{213}$ as well, where $\pi'=\overline{21}3$ is said to have one vinculum and $\pi''=\overline{213}$ has two vincula.
Notice that the vincula do not have to be adjacent themselves, as we can have a vincular pattern like $\pi=\overline{12}\,\overline{34}5\overline{678}$, which has four vincula.
Classical patterns can be seen as vincular patterns with no vincula.

Vincular pattern avoidance, avoidance classes, and Wilf equivalences are defined analogously.
In particular, these vincular notions apply to cyclic patterns and permutations as well, without change.
For example, Domagalski, Liang, Minnich, Sagan, Schmidt, and Sietsema \cite{domagalski2021cyclic} proved that $\left|\Av_n[13\overline{24}]\right|$, the number of cyclic permutations of $[n]$ avoiding $[13\overline{24}]$, equals $C_{n-1}$, where $C_n$ denotes the $n$th Catalan number.

We would like to note one difference between pattern avoidance in the vincular case, as opposed to the non-vincular case.
This observation applies to both linear and cyclic permutations.
Without vincula, we have the simple but oftentimes useful property that if $\sigma$ avoids a pattern $\pi$, then any subsequence $\sigma'$ of $\sigma$ also avoids $\pi$.
But this is not necessarily the case when $\pi$ is a vincular pattern, as removing elements from $\sigma$ changes the adjacency relations of the other elements.
For example, $\sigma = 1423$ avoids $\overline{12}3$, but the subsequence $\sigma' = 123$ contains $\overline{12}3$.
However, with some additional caution, this reasoning can sometimes still be applied: for an example, see the proof of \cref{theorem: vincular 4 one vinculum D}.

\section{Vincular cyclic patterns of length 3}\label{section: vincular 3}
Before we address vincular cyclic patterns of length 3, let us address the two vincular cyclic patterns of length 2, $[\overline{12}]$ and $[\overline{21}]$.
It is easy to see that every $[\sigma]$ of length $|\sigma| \geq 2$ contains both $[\overline{12}]$ and $[\overline{21}]$, for every cyclic permutation of length at least two must contain at least one ascent and one descent.
Hence, we have the following proposition.
\begin{proposition}\label{proposition: vincular 2}
We have $\left|\Av_n[\overline{12}]\right|=\left|\Av_n[\overline{21}]\right|=\begin{cases} 1 & n = 1 \\ 0 & n \geq 2.\end{cases}$
\end{proposition}

We now address vincular cyclic patterns of length 3 with a single vinculum.
\begin{theorem}\label{theorem: vincular 3 one vinculum}
For vincular cyclic patterns $\pi$ of length 3 containing one vinculum, we have $\left|\Av_n[\pi]\right|=1$ for all $n \geq 1$.
\end{theorem}
\begin{proof}
The vincular cyclic patterns of length 3 with one vinculum are 
\[[\overline{12}3]\equiv[\overline{21}3]\equiv[\overline{23}1]\equiv[\overline{32}1]\]
and
\[[\overline{13}2]\equiv[\overline{31}2],\]
where these Wilf equivalences are trivial Wilf equivalences.

For $n < 3$, we have only one cyclic permutation, so the result holds.
So we now assume $n \geq 3$.

We first address $[\overline{12}3]$.
Consider $[\sigma]\in[S_n]$ that avoids $[\overline{12}3]=[3\overline{12}]$.
As 1 must be followed by an ascent, this ascent must be to $n$, as otherwise 1, the element just after it, and $n$ form a $[\overline{12}3]$ pattern.
Consider 2 and the element just after it.
This element cannot be $n$ as 1 is just before $n$, but if it is an ascent then we have a $[\overline{12}3]$ pattern using $n$ as our 3, so 2 must be just before 1.
Continuing this logic, we find that $[\sigma]=[(n-1),\dots,2,1,n]=[\delta_n]$ is the only $[\overline{12}3]$-avoiding cyclic permutation, and hence $\left|\Av_n[\overline{12}3]\right|=1$.

Now we address $[\overline{31}2]$.
Consider $[\sigma]\in[S_n]$ that avoids $[\overline{31}2]=[2\overline{31}]$.
Note that 1 and the element just before it must form a descent, and so to avoid $[\overline{31}2]$ the element just before 1 must be a 2, as otherwise these two elements along with 2 form a $[\overline{31}2]$ pattern.
Similarly, 2 and the element just before it must also form a descent, and we find 3 must be just before 2.
Inductively, we find that $[\sigma]=[\delta_n]$ is the only possibility, so $\left|\Av_n[\overline{31}2]\right|=1$.
\end{proof}

Next, we address vincular cyclic patterns of length 3 with two vincula.
There are six such vincular cyclic patterns, namely $[\overline{123}]\equiv[\overline{321}]$ and $[\overline{132}]\equiv[\overline{213}]\equiv[\overline{231}]\equiv[\overline{312}]$, where these Wilf equivalences are trivial.
Viewing each cyclic $[\sigma]$ permutation as starting at $\sigma_1=1$, in order to avoid $[\overline{123}]$, or two consecutive cyclic ascents, as we cannot ascend to 1, this is equivalent to avoiding a double ascent in the (linear) permutation $\sigma_2\cdots\sigma_n$ as well as avoiding an initial ascent, i.e., $\sigma_2 < \sigma_3$.
Bergeron, Flajolet, and Salvy \cite{bergeron1992varieties} showed this sequence gives the exponential generating function
\begin{equation}\label{eq: 123 egf}
    \sum_{n\geq 0} \left|\Av_{n+1}[\overline{123}]\right| \frac{z^n}{n!} = \frac{1}{2} + \frac{\sqrt{3}}{2}\tan\left(\frac{\sqrt{3}}{2}z+\frac{\pi}{6}\right),
\end{equation}
which satisfies the differential equation
\begin{equation}\label{eq: 123 diff eq}
    E' = E^2 - E + 1.
\end{equation}
This resolves the first half of \cite[Conjecture 6.4]{domagalski2021cyclic}, although we had to change the indices of the exponential generating function to use $\left|\Av_{n+1}[\overline{123}]\right|$ in order for the conjectured differential equation to hold.
Elizalde and Sagan \cite[Corollary 2]{elizalde2021consecutive}\footnote{This paper \cite{elizalde2021consecutive} will be merged with \cite{domagalski2021cyclic} in a later version.} concurrently and independently proved a more general result that implies \cref{eq: 123 diff eq}, which thus has the explicit solution given by \cref{eq: 123 egf}; their proof method relates $[\overline{123}]$-avoiding cyclic permutations to $\overline{123}$-avoiding linear permutations, contrasting with our method of relating it to linear permutations without double ascents or an initial ascent.
We note that a result by Ehrenborg \cite[Theorem 3.3]{ehrenborg2016cyclically} implies the following closed form for $\left|\Av_{n}[\overline{123}]\right|$:
\begin{equation*}
    \left|\Av_n[\overline{123}]\right| = (n-1)! \sum_{k=-\infty}^\infty \left(\frac{\sqrt{3}}{2\pi(k+1/3)}\right)^n.
\end{equation*}
% \Rupert{How to cite? Not sure who the exponential generating function is due to.}
% \Rupert{Also not sure who the DE is due to, and also technically Conjecture 6.4 uses $\Av_n$ instead of $\Av_{n+1}$.}
In order for $[\sigma]$ to avoid $[\overline{132}]$, assuming $\sigma_1 = 1$, $\sigma_1$ cannot be either the 3 or 2 in the vincular pattern $[\overline{132}]$, so this is equivalent to the linear permutation $\sigma$, which must start with 1 and avoid $\overline{132}$.
It is easy to see from the commented description that this is counted by the OEIS sequence A052319, which gives exponential generating function
\begin{equation}\label{eq: 132 egf}
    \sum_{n \geq 1} \left|\Av_n[\overline{132}]\right| \frac{z^n}{n!} = -\ln\left(1-\sqrt{\frac{\pi}{2}}\erf\left(\frac{z}{\sqrt{2}}\right)\right),
\end{equation}
where $\erf(z)$ is the error function
\begin{equation*}
    \erf(z)=\frac{2}{\sqrt{\pi}} \int_0^z e^{-t^2} dt.
\end{equation*}
This exponential generating function satisfies the differential equation
\begin{equation}\label{eq: 132 diff eq}
    E' = e^{\displaystyle E-z^2/2},
\end{equation}
resolving the second half of \cite[Conjecture 6.4]{domagalski2021cyclic}.
Elizalde and Sagan \cite[Corollary 4]{elizalde2021consecutive} also concurrently and independently resolved this second half of the conjecture by proving a more general result that implies \cref{eq: 132 diff eq}, which in turn implies \cref{eq: 132 egf}.
% \Rupert{Again, how should I cite this?}

\section{Multiple vincular cyclic patterns of length 3}\label{section: multiple vincular 3}
Similar to the study of \cite{domagalski2021cyclic} for non-vincular cyclic patterns, we analyze the sizes of the avoidance classes for sets of multiple vincular cyclic patterns.

Consider a set $\Pi$ of (potentially vincular) cyclic patterns of length 3.
By \cref{theorem: vincular 3 one vinculum}, if $\Pi$ contains a vincular pattern with one vinculum, then $\left|\Av_n[\Pi]\right| \leq 1$, where if $\Av_n[\Pi]$ is nonempty then either $\Av_n[\Pi]=\{[\iota_n]\}$ or $\Av_n[\Pi]=\{[\delta_n]\}$.
Moreover, the non-vincular cyclic patterns of length 3 are $[321]$ and $[123]$, which are only avoided by $[\iota_n]$ and $[\delta_n]$, respectively.
We see that $[\iota_n]$ avoids the set of vincular patterns $\{[321],[\overline{13}2],[\overline{21}3], [\overline{32}1],[\overline{132}],[\overline{213}],[\overline{321}]\}$, i.e., the set of possibly vincular cyclic patterns of length 3 that when ``de-vincularized" are equal to $[321]$, and $[\delta_n]$ avoids $\{[123],[\overline{12}3],[\overline{23}1],[\overline{31}2],[\overline{123}],[\overline{231}],[\overline{312}]\}$, the set of possibly vincular cyclic patterns of length 3 that when de-vincularized are equal to $[123]$; in addition, each contains all the patterns in the other set.
Hence, assuming $\Pi$ contains a vincular pattern with at most one vinculum, if $\Pi$ is a subset of either of these two sets, then $\left|\Av_n[\Pi]\right| = 1$, namely containing the respective $[\iota_n]$ or $[\delta_n]$.
Otherwise, $\left|\Av_n[\Pi]\right|=0$ for $n \geq 3$.

It remains to consider $\Pi$ containing only vincular patterns of length 3 with two vincula.
This case is far more interesting, as there are many more cyclic permutations avoiding such patterns.

\subsection{Doubleton sets}\label{subsection: multiple vincular 3 doubleton}
There are $\binom{6}{2}=15$ doubleton sets, i.e., sets containing two elements, that contain vincular patterns of length 3 with two vincula.
We first address the doubleton sets that do not admit any cyclic permutations.
\begin{proposition}\label{proposition: multiple vincular 3 doubleton 0}
For $\Pi$ equaling any of the following six doubleton sets
\[ \{[\overline{123}],[\overline{132}]\},\,\,\,
\{[\overline{123}],[\overline{213}]\},\,\,\,
\{[\overline{321}],[\overline{231}]\},\,\,\,
\{[\overline{321}],[\overline{312}]\},\,\,\,
\{[\overline{132}],[\overline{231}]\},\,\,\,
\{[\overline{213}],[\overline{312}]\}, \]
we have $\left|\Av_n[\Pi]\right|=\begin{cases} 1 & n \leq 2 \\ 0 & n \geq 3.\end{cases}$
\end{proposition}
\begin{proof}
For $n \leq 2$, there is one cyclic permutation, which yields the desired result.

All six of these doubleton sets consist of two patterns that either share a 1 in the same position or a 3 in the same position.
Consider the corresponding consecutive subsequence of three elements that uses 1 or $n$ in the corresponding position, if the shared value is 1 or 3, respectively.
Then the other two values $x$ and $y$ must either satisfy $x<y$ or $x>y$, which will not avoid one of the patterns.

For example, consider the first set $\{[\overline{123}],[\overline{132}]\}$.
Consider an arbitrary cyclic permutation of length $n\geq 3$, and consider the consecutive subsequence of three elements that starts with $1$, i.e., of the form $1xy$.
If $x<y$, then this is order isomorphic to $[\overline{123}]$, and otherwise it is order isomorphic to $[\overline{132}]$, so no permutation can avoid both patterns.
\end{proof}

We now address the case of $\Pi=\{[\overline{123}],[\overline{321}]\}$, which we show is equinumerous with the up-down permutations.
Recall that an up-down permutation on $n$ elements is a permutation $\sigma=\sigma_1\cdots \sigma_n$ where $\sigma_1 < \sigma_2 > \sigma_3 < \sigma_4 > \cdots$, i.e., the permutation alternates between ascents and descents.
Let $U_n$ be the number of up-down permutations on $n$ elements.
Up-down permutations are also referred to as alternating permutations, however use of this terminology is inconsistent as other authors define alternating permutations to also include the down-up permutations, defined analogously.
\begin{proposition}\label{proposition: multiple vincular 3 doubleton alternating}
For all $n \geq 1$,
\begin{equation*}
    \left|\Av_n[\overline{123},\overline{321}]\right| =
    \begin{cases}
    1 & n = 1 \\
    0 & n \geq 3 \text{ is odd} \\
    U_{n-1} & n \geq 2 \text{ is even}.
    \end{cases}
\end{equation*}
\end{proposition}
\begin{proof}
For $n \leq 2$, we have one cyclic permutation, so the result holds; in particular, $U_1 = 1$.
For $n \geq 3$, notice that in order to avoid two consecutive cyclic ascents and two consecutive cyclic descents, the cyclic permutation must be alternating, i.e., alternate between cyclic ascents and descents.
This is clearly impossible for odd $n$ due to the cyclic nature of the permutation, so no permutations avoid both vincular cyclic patterns.

It remains to resolve the even case for $n \geq 4$.
Consider a cyclic permutation $[\sigma]$ that avoids $\{[\overline{123}],[\overline{321}]\}$ of length $n$.
Without loss of generality we may assume $\sigma_1 = n$.
In particular, we must descend from $n$, and ascend to $n$, so the linear permutation $\sigma_2\cdots\sigma_n$ is an up-down permutation of length $n-1$.
This is a necessary and sufficient condition for $[\sigma]$ to avoid $\{[\overline{123}],[\overline{321}]\}$.
Hence, there are $U_{n-1}$ such cyclic permutations.
\end{proof}
\begin{remark}\label{remark: multiple vincular 3 doubleton alternating tangent}
Up-down permutations were first studied by Andr\'e \cite{andre1881permutations}, who showed that the exponential generating function of the number of up-down permutations $U_n$ is $\tan(x)+\sec(x)$, where $\tan(x)$ provides the terms with odd degree and $\sec(x)$ provides the terms with even degree.
Thus, including the $n=1$ term, we find
\begin{equation*}
    \sum_{n\geq 0} \left|\Av_{n+1}[\overline{123},\overline{321}]\right| \frac{z^n}{n!}= 1 + \tan(z),
\end{equation*}
or equivalently
\begin{equation*}
    \sum_{n \geq 1} \left|\Av_n[\overline{123},\overline{321}]\right| \frac{z^n}{n!} = \int_0^z (1 + \tan(x)) dx = z - \ln(\cos(z)).
\end{equation*}
The asymptotics of the proportion of permutations that are up-down is
\begin{equation*}
    \frac{U_n}{n!} =  2\left(\frac{2}{\pi}\right)^{n+1} + O\left(\left(\frac{2}{3\pi}\right)^n\right);
\end{equation*}
see Stanley \cite{stanley2010survey}.
\end{remark}

The remaining doubleton sets form three Wilf equivalence classes under the trivial Wilf equivalences:
\begin{align*}
    \{[\overline{123}],[\overline{231}]\} \equiv \{[\overline{123}],[\overline{312}]\} & \equiv \{[\overline{321}],[\overline{132}]\} \equiv \{[\overline{321}],[\overline{213}]\} \tag{A} \\
    \{[\overline{132}],[\overline{213}]\} &\equiv \{[\overline{231}],[\overline{312}]\} \tag{B} \\
    \{[\overline{132}],[\overline{312}]\} &\equiv \{[\overline{213}],[\overline{231}]\} \tag{C}.
\end{align*}
A computer search demonstrates none of these three classes are Wilf equivalent, and provides the following table of data, \cref{tab: multiple vincular 3 doubleton}, on the number of cyclic permutations avoiding a member of one of these three Wilf equivalence classes.
Currently, none of these three sequences appear in the OEIS \cite{sloane2018line}.
We leave the enumeration of these three Wilf equivalence classes as an open problem.
\begin{table}[htbp!]
    \centering
    \begin{tabular}{|l|l|l|l|}
        \hline
        $n$ & $\left|\Av_n[(\mathrm{A})]\right|$ & $\left|\Av_n[(\mathrm{B})]\right|$ & $\left|\Av_n[(\mathrm{C})]\right|$ \\
        \hline
        1 & 1 & 1 & 1 \\
        2 & 1 & 1 & 1 \\
        3 & 1 & 1 & 0 \\
        4 & 1 & 1 & 1 \\
        5 & 4 & 3 & 2 \\
        6 & 14 & 12 & 6 \\
        7 & 54 & 46 & 20 \\
        8 & 278 & 218 & 86 \\
        9 & 1524 & 1206 & 416 \\
        10 & 9460 & 7272 & 2268 \\
        11 & 66376 & 49096 & 13598 \\
        12 & 504968 & 366547  & 89924  \\
        13 & 4211088 & 2970945  & 649096  \\
        \hline
    \end{tabular}
    \vspace{2mm}
    \caption{Size of avoidance classes of doubleton sets of vincular cyclic patterns of length 3.}
    \label{tab: multiple vincular 3 doubleton}
\end{table}

\subsection{Three or more patterns}\label{subsection: multiple vincular 3 three+}
As observed previously in \cref{section: multiple vincular 3}, if a set $\Pi$ of cyclic patterns of length 3 contains a pattern with at most one vinculum, then we have $\left|\Av_n[\Pi]\right| \leq 1$ and its exact value can be easily determined.
Thus we will only consider $\Pi$ consisting of three or more vincular patterns of length 3, all with two vincula.
There are $\binom{6}{3}=20$ sets of three such patterns.
All but two of these contain a set from \cref{proposition: multiple vincular 3 doubleton 0}, so it follows that for $n \geq 3$, no cyclic permutations in $[S_n]$ avoid all three patterns in each of these 18 sets.
The two remaining sets are
\begin{equation*}
    \{[\overline{123}],[\overline{231}],[\overline{312}]\} \equiv \{[\overline{132}],[\overline{213}],[\overline{321}]\}.
\end{equation*}
Adding any other such pattern will yield one of the sets from \cref{proposition: multiple vincular 3 doubleton 0} contained in $\Pi$, so we see that for any $\Pi$ consisting of more than three patterns, $\left|\Av_n[\Pi]\right|=0$ for $n \geq 3$, and 1 otherwise.

We now provide a bijection between $\Av_n[\overline{132},\overline{213},\overline{321}]$ and the set of total cyclic orders extending a particular partial cyclic order.
We first define the necessary concepts regarding partial cyclic orders, which can be seen as a circular analog of a poset.
We refer readers to Megiddo \cite{megiddo1976partial} for a more comprehensive introduction of cyclic orders.
\begin{definition}\label{definition: partial cyclic order}
A partial cyclic order on a set $X$ is a ternary relation $Z \subset X^3$ satisfying the following conditions:
\begin{enumerate}
    \item $\forall x,y,z \in X, (x,y,z) \in Z \Rightarrow (y,z,x) \in Z$ (cyclicity),
    \item $\forall x,y,z \in X, (x,y,z) \in Z \Rightarrow (z,y,x) \not\in Z$ (antisymmetry),
    \item $\forall x,y,z,u \in X, (x,y,z) \in Z$ and $(x,z,u) \in Z \Rightarrow (x,y,u) \in Z$ (transitivity).
\end{enumerate}
A partial cyclic order $Z$ on a set $X$ is a total cyclic order if for any triple $(x,y,z)$ of distinct elements, either $(x,y,z)\in Z$ or $(z,y,x)\in Z$.
\end{definition}
A partial cyclic order $Z'$ extends another partial cyclic order $Z$ if $Z\subseteq Z'$.
The problem of determining whether a given partial cyclic order admits a cyclic extension is NP-complete, as shown by Galil \& Megiddo \cite{galil1977cyclic}.

One way in which cyclic orders can be seen as a circular analog of partial orders is that a totally ordered set can be organized into a chain, where an element $y$ is larger than $x$ if $y$ is above $x$; on the other hand, a total cyclic order can be visually represented by placing the elements of $X$ on a circle, so that $(x,y,z)\in Z$ if and only if, starting from $x$ and going in the positive (i.e., counterclockwise) direction, one encounters $y$ before encountering $z$.

We now use a simplification of the notation of Ramassamy \cite{ramassamy2018extensions}, where we use $\mathcal{R}_n$ in place of $\mathcal{R}_{+^n}^{+,+}$ as used in his paper.
\begin{definition}
For any positive integer $n$, let $\mathcal{R}_n$ denote the set of total cyclic orders $Z$ on the set $X=[n+2]$ such that $(i,i+1,i+2)\in Z$ for all $1 \leq i \leq n$, as well as $(n+1,n+2,1)\in Z$ and $(n+2,1,2)\in Z$.
\end{definition}
Ramassamy \cite{ramassamy2018extensions} proved a recurrence relation that enumerates $\left|\mathcal{R}_n\right|$.
This recurrence relation is quite complicated, however, so for sake of brevity and clarity we do not include this recurrence, and refer readers to \cite[Theorem 3.5]{ramassamy2018extensions}.
No closed form or nice expression for a generating function is currently known for this sequence of numbers $\left|\mathcal{R}_n\right|$, which is OEIS sequence A295264 \cite{sloane2018line}.
However, the recurrence does allow for the construction of an algorithm that calculates the first $n$ values of this sequence in polynomial time, compared to the super-exponential complexity associated with a complete search over all permutations.

We now provide a bijection to demonstrate $\left|\Av_{n+2}[\overline{132},\overline{213},\overline{321}]\right|=\left|\mathcal{R}_n\right|$.
\begin{theorem}\label{theorem: multiple vincular 3 tripleton cyclic order}
For all $n \geq 1$, we have
\begin{equation*}
    \left|\Av_{n+2}[\overline{132},\overline{213},\overline{321}]\right|=\left|\mathcal{R}_n\right|.
\end{equation*}
\end{theorem}
\begin{proof}
Consider an element $[\sigma]\in\Av_{n+2}[\overline{132},\overline{213},\overline{321}]$, where we may assume that $\sigma_1 = 1$.
We construct a bijection $\phi: \Av_{n+2}[\overline{132},\overline{213},\overline{321}] \to \mathcal{R}_n$.
The cyclic order $\phi([\sigma])\in \mathcal{R}_n$ has $(i,j,k)\in \phi([\sigma])$ if and only if $\sigma_i\sigma_j\sigma_k$ is order isomorphic to 123, 231, or 312.

We first verify $\phi([\sigma])\in \mathcal{R}_n$.
Clearly this satisfies cyclicity, as well as antisymmetry.
Proving transitivity is more involved.
Suppose $(x,y,z),(x,z,u)\in \phi([\sigma])$.
We split into three cases depending on whether $\sigma_x\sigma_y\sigma_z$ is order isomorphic to 123, 231, or 312.
\begin{enumerate}
    \item Suppose $\sigma_x \sigma_y \sigma_z\cong 123$.
    
    Then as $\sigma_x < \sigma_z$ and $(x,z,u)\in \phi([\sigma])$, either $\sigma_x\sigma_z\sigma_u\cong 123$ or $\sigma_x\sigma_z\sigma_u\cong 231$.
    In the former case, notice that this implies $\sigma_x \sigma_y \sigma_z \sigma_u\cong 1234$, and thus $\sigma_x\sigma_y\sigma_u \cong 123$, so $(x,y,u)\in \phi([\sigma])$.
    In the latter case, this implies $\sigma_x \sigma_y \sigma_z \sigma_u \cong 2341$, and thus $\sigma_x\sigma_y\sigma_u \cong 231$, so $(x,y,u)\in \phi([\sigma])$.
    
    \item Suppose $\sigma_x \sigma_y \sigma_z \cong 231$.
    
    Then in order to have $(x,z,u)\in \phi([\sigma])$, we must have $\sigma_x\sigma_z\sigma_u \cong 312$, so $\sigma_x\sigma_y\sigma_z\sigma_u \cong 3412$, and thus $\sigma_x\sigma_y\sigma_u \cong 231$, so $(x,y,u) \in \phi([\sigma])$.
    
    \item Suppose $\sigma_x\sigma_y\sigma_z\cong 312$.
    
    Then in order to have $(x,z,u)\in\phi([\sigma])$, we must have $\sigma_x\sigma_z\sigma_u\cong312$, so $\sigma_x\sigma_y\sigma_z\sigma_u\cong4123$, and thus $\sigma_x\sigma_y\sigma_u\cong312$, so $(x,y,u)\in\phi([\sigma])$.
\end{enumerate}
It is easy to see $\phi([\sigma])$ is a total cyclic order.
Lastly, we must show that $(i,i+1,i+2)\in \phi([\sigma])$ for all $1 \leq i \leq n$, as well as $(n+1,n+2,1),(n+2,1,2) \in \phi([\sigma])$.
As $[\sigma] \in \Av_{n+2}[\overline{132},\overline{213},\overline{321}]$, we have that any such triple of indices must avoid the vincular patterns $[\overline{132}],[\overline{213}]$, and $[\overline{321}]$, so in particular these three cyclically consecutive values must be order isomorphic to 123, 231, or 312, as desired.
Hence, $\phi([\sigma])\in\mathcal{R}_n$.

We now construct the inverse map $\psi: \mathcal{R}_n \to  \Av_{n+2}[\overline{132},\overline{213},\overline{321}]$, which we will then show satisfies $\phi\circ\psi=\operatorname{id}$.
Given a total cyclic order $Z \in \mathcal{R}_n$, construct $[\sigma]\in [S_{n+2}]$ with $\sigma_1=1$ by enforcing $\sigma_i < \sigma_j$ for distinct $i,j > 1$ if and only if $(1,i,j)\in Z$.
Notice that for any such $\sigma_i$ and $\sigma_j$, either $\sigma_i < \sigma_j$ or $\sigma_j < \sigma_i$ as $Z$ is a total cyclic order.
We claim that a unique assignment of the elements of $[2,n+2]$ to $\sigma_2,\dots,\sigma_{n+2}$ exists that satisfies all of these prescribed relations.
It suffices to show that a directed cycle of self-contradictory relations $\sigma_{i_1} < \sigma_{i_2} < \cdots < \sigma_{i_k} < \sigma_{i_1}$ cannot occur, as this implies the relations form a total order, which can be mapped order isomorphically to $[2,n+2]$.
Suppose for the sake of contradiction that such a directed cycle existed.
Then by construction of $\psi$, we have $(1,i_1,i_2),(1,i_2,i_3),\dots,(1,i_{k-1},i_k)\in Z$, and applying transitivity yields $(1,i_1,i_k)\in Z$, so by antisymmetry $(1,i_k,i_1)\not\in Z$.
But our cycle's final relation $\sigma_{i_k} < \sigma_{i_1}$ implies $(1,i_k,i_1)\in Z$, reaching a contradiction.
Hence a unique assignment of the elements of $[2,n+2]$ to $\sigma_2,\dots,\sigma_{n+2}$ exists, which yields a unique $[\sigma]\in [S_{n+2}]$.
We define $\psi(Z)=[\sigma]$.

We first show that $\psi(Z) \in \Av_{n+2}[\overline{132},\overline{213},\overline{321}]$.
Let $[\sigma]=\psi(Z)$.
Consider three cyclically consecutive elements $\sigma_i\sigma_{i+1}\sigma_{i+2}$ of $\psi(Z)$, where the indices are taken modulo $n+2$.
By definition of $\mathcal{R}_n$, we know $(i,i+1,i+2)\in Z$.
As $Z$ is a total cyclic order, we split into two cases depending on whether $(1,i,i+2)\in Z$ or not.

Case 1: $(1,i,i+2)\in Z$.
Then $(i,i+1,i+2),(i,i+2,1)\in Z$, which implies $(i,i+1,1)\in Z$, so $\sigma_i < \sigma_{i+1}$.
Moreover, $(i+2,1,i),(i+2,i,i+1)\in Z$, which implies $(i+2,1,i+1)\in Z$, so $\sigma_{i+1} < \sigma_{i+2}$, so $\sigma_i\sigma_{i+1}\sigma_{i+2}\cong 123$, and thus it avoids the three forbidden patterns.

Case 2: $(1,i+2,i) \in Z$.
Then $\sigma_{i+2}<\sigma_i$.
If $(1,i,i+1)\in Z$, then $\sigma_i < \sigma_{i+1}$, and thus $\sigma_i\sigma_{i+1}\sigma_{i+2}\cong 231$, which avoids the three forbidden patterns.
Otherwise, we have $(1,i+1,i)\in Z$.
Then we have $(i+1,i+2,i),(i+1,i,1)\in Z$, so transitivity implies $(i+1,i+2,1)\in Z$, and thus $\sigma_{i+1}<\sigma_{i+2}$.
Thus $\sigma_{i+1}<\sigma_{i+2}<\sigma_i$, so $\sigma_i\sigma_{i+1}\sigma_{i+2}\cong 312$, which avoids the three forbidden patterns.

Hence, we find $\psi(Z) \in \Av_{n+2}[\overline{132},\overline{213},\overline{321}]$.

Finally, it suffices to show that $\phi\circ\psi(Z)=Z$ for all $Z \in \mathcal{R}_n$.
Let $[\sigma]=\psi(Z)$.
Suppose $\sigma_i\sigma_j\sigma_k$ is order isomorphic to 123, 231, or 312.
It suffices to show that $(i,j,k)\in Z$.
By applying cyclicity afterwards, it suffices to address the case $\sigma_i\sigma_j\sigma_k\cong 123$.
As $\sigma_i < \sigma_j < \sigma_k$, by definition of $\psi$ we know $(1,i,j),(1,j,k)\in Z$.
So $(j,k,1)$ and $(j,1,i)$ are both in $Z$, and transitivity yields $(j,k,i)\in Z$, or equivalently $(i,j,k)\in Z$, as desired.
Hence $\phi\circ\psi(Z)=Z$, and $\phi$ is a bijection between $\Av_{n+2}[\overline{132},\overline{213},\overline{321}]$ and $\mathcal{R}_n$, which proves the result.
\end{proof}

This completes the classification for all sets $\Pi$ containing at least three vincular patterns of length 3.

\section{Vincular cyclic patterns of length 4}\label{section: vincular 4}
There are four ways to place vincula into a vincular pattern of length 4: one vinculum $[\overline{ab}cd]$, two disjoint vincula $[\overline{ab}\,\overline{cd}]$, two consecutive vincula $[\overline{abc}d]$, and three vincula $[\overline{abcd}]$.
We investigate the case of a single vinculum.

% \subsection{One vinculum}\label{subsection: vincular 4 one vinculum}
There are $4!=24$ vincular cyclic patterns of length 4 with one vinculum, grouped into the following equivalence classes via trivial Wilf equivalences.
\begin{align*}
    [\overline{12}34] \equiv [\overline{21}43] &\equiv [\overline{34}12] \equiv [\overline{43}21] \tag{A} \\
    [\overline{12}43] \equiv [\overline{21}34] &\equiv [\overline{34}21] \equiv [\overline{43}12] \tag{B} \\
    [\overline{13}24] \equiv [\overline{24}13] &\equiv [\overline{31}42] \equiv [\overline{42}31] \tag{C} \\
    [\overline{13}42] \equiv [\overline{24}31] &\equiv [\overline{31}24] \equiv [\overline{42}13] \tag{D} \\
    [\overline{14}23] &\equiv [\overline{41}32] \tag{E} \\
    [\overline{14}32] &\equiv [\overline{41}23] \tag{F} \\
    [\overline{23}14] &\equiv [\overline{32}41] \tag{G} \\
    [\overline{23}41] &\equiv [\overline{32}14] \tag{H}
\end{align*}

A computer search provides the following table of data, \cref{tab: vincular 4 one vinculum}, on the number of cyclic permutations in each trivial Wilf equivalence class of avoidance classes.
\begin{table}[htbp!]
    \centering
    \begin{tabular}{|l|l|l|l|l|l|l|l|l|l|l|l|l|}
        \hline
        $n$ & $\left|\Av_n[(\mathrm{A})]\right|$ & $\left|\Av_n[(\mathrm{C})]\right|=\left|\Av_n[(\mathrm{E})]\right|$ & $\left|\Av_n[(\mathrm{D})]\right|$ & $\left|\Av_n[(\mathrm{G})]\right|$ & $\left|\Av_n[(\mathrm{H})]\right|$ \\
        & $=\left|\Av_n[(\mathrm{B})]\right|$ & $=\left|\Av_n[(\mathrm{F})]\right|=C_{n-1}$ & & & \\
        \hline
        1 & 1 & 1 & 1 & 1 & 1 \\
        2 & 1 & 1 & 1 & 1 & 1 \\
        3 & 2 & 2 & 2 & 2 & 2 \\ 
        4 & 5 & 5 & 5 & 5 & 5 \\ 
        5 & 14 & 14 & 13 & 14 & 15 \\ 
        6 & 43 & 42 & 35 & 42 & 50 \\ 
        7 & 144 & 132 & 97 & 133 & 180 \\ 
        8 & 523 & 429 & 275 & 442 & 690 \\ 
        9 & 2048 & 1430 & 794 & 1537 & 2792 \\
        10 & 8597 & 4862 & 2327 & 5583 & 11857 \\
        11 & 38486 & 16796 & 6905 & 21165 & 52633 \\
        12 & 182905 & 58786 & 20705 & 83707 & 243455 \\
        \hline
        OEIS & A047970 & A000108 & A025242 & A34666$0^*$ & A34666$1^*$ \\
        \hline
    \end{tabular}
    \vspace{2mm}
    \caption{Size of avoidance classes of vincular cyclic patterns of length 4 with one vinculum.
    Trivial Wilf equivalence classes that are enumerated by the same values for $n \leq 12$ are condensed into the same column; all of these nontrivial Wilf equivalences are proven in this paper.
    OEIS references are included; the last two, marked with asterisks, are new and come from this work.}
    \label{tab: vincular 4 one vinculum}
\end{table}

We first enumerate (A) and (B), demonstrating that they are Wilf equivalent.
In order to do this, we first define the Zeilberger statistic and set of a permutation, and then of a cyclic permutation.
\begin{definition}
The \emph{Zeilberger set}\footnote{The terminology of a Zeilberger set was suggested by Colin Defant.} of a permutation $\sigma \in S_n$ is the longest subsequence of the form $n,n-1,\dots,i$ for some $i$.
The \emph{Zeilberger statistic} of a permutation $\sigma$, denoted $\zeil(\sigma)$, is the length of the Zeilberger set, i.e., the largest integer $m$ such that $n,n-1,\dots,n-m+1$ is a subsequence of $\sigma$.
\end{definition}
The Zeilberger statistic originated in Zeilberger's study of stack-sortable permutations \cite{zeilberger1992proof}, and has been studied in articles such as \cite{bousquet1998multi,kitaev2008classification,bouvel2014refined}.
We extend this definition to cyclic permutations.
\begin{definition}
The \emph{Zeilberger set} of a cyclic permutation $[\sigma]\in[S_n]$ is the longest subsequence of the form $n,n-1,\dots,i$ for some $i$ appearing within some permutation $\sigma'\in[\sigma]$, i.e., some rotation of $\sigma$.
The \emph{Zeilberger statistic} of a cyclic permutation $[\sigma]$, denoted $\zeil[\sigma]$, is the cardinality of the Zeilberger set of $[\sigma]$, i.e., the largest integer $m$ such that $n,n-1,\dots,n-m+1$ is a subsequence of some rotation of $\sigma$.
Notice that $\zeil[\sigma]$ is the maximum of $\zeil(\sigma')$ for all rotations $\sigma'$ of $\sigma$.
\end{definition}
For example, the Zeilberger set of $[136254]$ is the subsequence 6543 because it is a subsequence of the rotation 625413.
\begin{definition}
The \emph{reverse Zeilberger set} of a cyclic permutation $[\sigma]\in[S_n]$ is the longest subsequence of the form $i,i+1,\dots,n$ for some $i$ appearing within some permutation $\sigma'\in[\sigma]$.
\end{definition}
Notice that the reverse Zeilberger set of $[\sigma]$ corresponds to the Zeilberger set of $[\sigma^r]$, and the \emph{reverse Zeilberger statistic}, the cardinality of the reverse Zeilberger set, is simply $\zeil[\sigma^r]$, so no new notation is needed.

\begin{theorem}\label{theorem: vincular 4 one vinculum A B}
For all $n \geq 2$, we have
\begin{equation*}
    \left|\Av_n[\overline{12}34]\right| = \left|\Av_n[\overline{12}43]\right| = 1 + \sum_{i=0}^{n-2} i (i+1)^{n-i-2}.
\end{equation*}
\end{theorem}
\begin{proof}
For $2 \leq n \leq 3$, we directly verify the result, where all cyclic permutations of length $n$ avoid both $[\overline{12}34]$ and $[\overline{12}43]$.

For general $n \geq 4$, we first address $[\overline{12}43]$.
We claim that the following criterion is a necessary and sufficient condition for a cyclic permutation $[\sigma]\in[S_n]$ to avoid $[\overline{12}43]$: for any cyclic ascent in $[\sigma]$, i.e., two adjacent elements $\sigma_i < \sigma_{i+1}$ with indices taken modulo $n$, say from $a$ to $b>a$, we have that $b$ must be in the reverse Zeilberger set of $[\sigma]$.
To see that this is sufficient, if $[\sigma]$ satisfies this criterion, then for any cyclic ascent from $a$ to $b$, as $b$ itself is within the reverse Zeilberger set, reading from $b$, we encounter the elements strictly greater than $b$ in increasing order, and thus we cannot get a copy of $[\overline{12}43]$ using this cyclic ascent as our $\overline{12}$.
As $\overline{12}$ must come from a cyclic ascent and our cyclic ascent was chosen arbitrarily, we find $[\sigma]$ avoids $[\overline{12}43]$.
To see that this criterion is necessary, if there exists a cyclic ascent in $[\sigma]$, say from $a$ to $b>a$, such that $b$ is not in the reverse Zeilberger set of $[\sigma]$, then reading $[\sigma]$ starting from $b$, we cannot encounter the elements strictly greater than $b$ in increasing order.
This means there exist $c$ and $d$ where $b<c<d$ and $d$ is encountered before $c$, from which we find $abdc$ forms a copy of $[\overline{12}43]$.

Using this equivalent criterion, we determine $\left|\Av_n[\overline{12}43]\right|$ by casework on the reverse Zeilberger statistic.
Clearly, for any $[\sigma]\in[S_n]$, we have $2 \leq \zeil[\sigma^r] \leq n$.
Suppose $\zeil[\sigma^r]=i+1$ for some $1 \leq i \leq n-1$.
If $i=n-1$, then $[\sigma]=[\iota_n]$, which satisfies the criterion.
For $1 \leq i \leq n-2$, the elements $n-i,\dots,n$, in that order, separate $[\sigma]$ into $i+1$ regions between these $i+1$ elements, in which all other elements $1,\dots,n-i-1$ must be placed.
As $n-i-1$ is not in the reverse Zeilberger set, it cannot be placed in the region between $n$ and $n-i$, so there are $i$ options on which region it can be placed into.
All other $n-i-2$ elements can be placed into any of the $i+1$ regions, yielding a total of $i(i+1)^{n-i-2}$ assignments.
In each region, the elements must be in decreasing order so that all cyclic ascents have the larger element in the reverse Zeilberger set.
This is necessary and sufficient to satisfy the criterion, and each of the $i(i+1)^{n-i-2}$ assignments yields a unique permutation with this property.
Summing over all $1 \leq i \leq n-2$ gives
\[ \left|\Av_n[\overline{12}43]\right| = 1 + \sum_{i=1}^{n-2} i(i+1)^{n-i-2},\]
as desired.

The proof for $[\overline{12}34]$ is essentially identical to the $[\overline{12}43]$ argument, where we instead replace the reverse Zeilberger set with the Zeilberger set.
\end{proof}

We note that $\left|\Av_{n+2}[\overline{12}34]\right|=1+\sum_{i=1}^n i(i+1)^{n-i}$ has a pattern avoidance interpretation using barred patterns; see Pudwell \cite{pudwell2010enumeration}.

The size of the avoidance class for class (C) was found in \cite{domagalski2021cyclic} to be the Catalan numbers: $\left|\Av[\overline{13}24]\right|=C_{n-1}$.

We now enumerate Wilf equivalence class (D), but first we provide the necessary definitions for this sequence.
The sequence $D_n$, studied previously in \cite{mansour2011restricted,sapounakis2007counting}, is given by the Catalan-like recurrence
\begin{equation*}
    D_n = \sum_{k=1}^{n-3} D_k D_{n-k}
\end{equation*}
for $n \geq 4$, with $D_1 = 2$, $D_2 = 1$, and $D_3 = 1$.
The next few values are $D_4 = 2$, $D_5 = 5$, and $D_6 = 13$.
For $n > 1$, the number $D_n$ counts the number of Dyck paths of semilength $n-1$ that avoid UUDD; see Sapounakis, Tasoulas, and Tsikouras \cite{sapounakis2007counting}.
Sapounakis et al.\ also proved the following expression for $D_{n+1}$:
\begin{equation*}
    D_{n+1} = \sum_{j=0}^{\left\lfloor n/2 \right\rfloor} \frac{(-1)^j}{n-j} \binom{n-j}{j}\binom{2n-3j}{n-j-1}.
\end{equation*}
Mansour and Shattuck \cite{mansour2011restricted} determined this sequence's ordinary generating function is
\begin{equation*}
    \sum_{n \geq 1} D_n z^n = \frac{(z+1)^2-\sqrt{1-4z+2z^2+z^4}}{2}.
\end{equation*}
We show that the size of the avoidance class for class (D) is enumerated by the sequence $D_n$.
\begin{theorem}\label{theorem: vincular 4 one vinculum D}
For all $n \geq 1$, we have $\left|\Av_n[\overline{13}42]\right| = D_{n+1}$.
\end{theorem}
\begin{proof}
The result clearly holds for $n \leq 4$.
For $n > 4$, we must show that
\begin{align*}
    \left|\Av_n[\overline{13}42]\right| &= \sum_{k=1}^{n-2} \left|\Av_{k-1}[\overline{13}42]\right|\left|\Av_{n-k}[\overline{13}42]\right|
    \\&= 2\left|\Av_{n-1}[\overline{13}42]\right|+\sum_{k=1}^{n-3} \left|\Av_k[\overline{13}42]\right|\left|\Av_{n-k-1}[\overline{13}42]\right|,
\end{align*}
where $\left|\Av_0[\overline{13}42]\right|$ is defined to be $D_1=2$.
Consider $[\sigma]$ of length $n > 4$ that avoids $[\overline{13}42]$, where we assume $\sigma_1 = 1$.
Let $\sigma_2 = m$, where $2 \leq m \leq n$.
We will proceed by casework on the value of $m$, where each of the $n-1$ values of $m$ corresponds to one of the $n-3$ terms in the desired sum, or one of the two copies of $\left|\Av_{n-1}[\overline{13}42]\right|$.
In this casework, it will be important to reference specific elements of the cyclic permutation $[\sigma]$, so for clarity we will use indices with respect to the linear permutation $\sigma$, which is well-defined from $[\sigma]$ as assuming $\sigma_1=1$ fixes the particular rotation, i.e., the particular element of $[\sigma]$, that we use.
However, we still must avoid the cyclic pattern $[\overline{13}42]$, so when concerning ourselves with pattern avoidance we will return to discussing the cyclic permutation $[\sigma]$, rather than the linear permutation $\sigma$.

If $m=2$, then the only additional $[\overline{13}42]$ patterns that could arise from removing 1 would be when the vinculum $\overline{13}$ bridges over the removed 1, i.e., when 2 plays the role of 3 in the pattern $\overline{13}42$.
However, as removing 1 would make 2 the smallest element, this is not possible.
So removing 1 and reducing takes $[\sigma]$ to a $[\overline{13}42]$-avoiding cyclic permutation of length $n-1$.
Conversely, increasing the values of all elements in a $[\overline{13}42]$-avoiding cyclic permutation of length $n-1$ by one and then inserting a 1 before the 2 yields a $[\overline{13}42]$-avoiding cyclic permutation of length $n$, as the only additional $[\overline{13}42]$ patterns that could potentially arise from this insertion would require 1 to be part of the pattern, in which it would have to play the role of 1, but then 2 could not play the role of 3 for no element has value strictly between 1 and 2.
Thus, there is a natural bijection from the $m=2$ subcase to $\Av_{n-1}[\overline{13}42]$ via removing 1 and reducing, i.e., subtracting 1 from each element, which yields $\left|\Av_{n-1}[\overline{13}42]\right|$ possibilities.

Similarly, if $m = n$, then clearly 1 cannot be part of a $[\overline{13}42]$ pattern, and removing 1 cannot create additional $[\overline{13}42]$ patterns with $n$ playing the role of 3, as $n$ is the largest element.
Hence there is a natural bijection from the $m=n$ subcase to $\Av_{n-1}[\overline{13}42]$ via removing 1 and reducing.
This yields the second copy of $\left|\Av_{n-1}[\overline{13}42]\right|$ needed in our recursion.

We now address $3 \leq m \leq n-1$.
Notice that we must have all elements in $[2,m-1]$ come before all the elements of $[m+1,n]$ in $\sigma$, so $\sigma$ is of the form $1m\rho\tau$ where $\rho$ is a permutation of $[2,m-1]$ and $\tau$ is a permutation of $[m+1,n]$.
To avoid a $[\overline{13}42]$ pattern in $[\sigma]$ where the last element of $\rho$ plays the role of 1, we require $\tau_1=n$, as otherwise we can use the last element of $\rho$ as the 1, $\tau_1$ as the 3, $n$ as the 4, and $m$ as the 2 in an occurrence of $[\overline{13}42]$.
So $\sigma$ is of the form $1m\rho n \tau'$ where $\tau'$ is a permutation of $[m+1,n-1]$.
Clearly if $[\sigma]$ is $[\overline{13}42]$-avoiding, then so too are $[m\rho]$ and $[n\tau']$, as the additional ascent from the last element of $\rho$ to $m$ cannot be the $\overline{13}$ in a $\overline{13}42$ pattern, as there is no element higher than $m$ in $m\rho$ to be the 4 in the pattern, and similarly the additional ascent from the last element of $\tau'$ to $n$ cannot be the $\overline{13}$ in a $\overline{13}42$ pattern.

We now show that if $[m\rho]$ and $[n\tau']$ are both $[\overline{13}42]$-avoiding for $\rho$ a permutation of $[2,m-1]$ and $\tau'$ a permutation of $[m+1,n-1]$, then $[1m\rho n \tau']\in[S_n]$ is also $[\overline{13}42]$-avoiding.
We proceed by casework on the position of the cyclic ascent playing the role of $\overline{13}$ in a potential $[\overline{13}42]$ pattern in $[1m\rho n \tau']$, and show we cannot find a pair of elements to be the 4 and the 2.
If the cyclic ascent is 1 to $m$, then as all elements smaller than $m$ are before all the elements greater than $m$ in $1m\rho n \tau'$, we cannot form a $[\overline{13}42]$ pattern.
The pair of adjacent elements $m$ and $\rho_1$ forms a cyclic descent, so we continue onward and consider a cyclic ascent in $\rho$.
If a $[\overline{13}42]$ pattern existed in $[1m\rho n\tau']$ where the $\overline{13}$ is within $\rho$, then clearly the 2 in the pattern must also come from $\rho$.
If the 4 also comes from $\rho$, then $[m\rho]$ would have contained $[\overline{13}42]$, contradicting the assumption that $[m\rho]$ avoids $[\overline{13}42]$.
Otherwise, the 4 comes from outside $\rho$, and notice that using $m$ in place of this element still yields a $[\overline{13}42]$ pattern, which again implies $[m\rho]$ contains $[\overline{13}42]$, a contradiction.
The cyclic ascent from the last element of $\rho$ to $n$ cannot be the $\overline{13}$ in a $[\overline{13}42]$ pattern, as there is no element larger than $n$ to be the 4.
We have a cyclic descent from $n$ to $\tau'_1$, as well as a cyclic descent from the last element of $\tau'$ to 1, so the last case to consider is a cyclic ascent within $\tau'$ being the $\overline{13}$.
If a $[\overline{13}42]$ pattern existed in $[1m\rho n\tau']$ where the $\overline{13}$ is within $\tau'$, then clearly the 2 in the pattern must also come from $\tau'$.
If the 4 also comes from $\tau'$, then this same subsequence is in $[n\tau']$, contradicting the assumption that $[n\tau']$ avoids $[\overline{13}42]$.
Otherwise, the 4 comes from outside $\tau'$, and notice that using $n$ in place of this element still yields a $[\overline{13}42]$ pattern, which again implies $[n\tau']$ contains $[\overline{13}42]$, a contradiction.
Thus, $[1m\rho n\tau']$ avoids $[\overline{13}42]$.

In particular, each $[\overline{13}42]$-avoiding permutation of $[2,m]$ has a unique linear representation as $m\rho$ and each $[\overline{13}42]$-avoiding permutation of $[m+1,n]$ has a unique linear representation as $n\tau'$, which is crucial as our expression of $\sigma$ as $1m\rho n\tau'$ is a linear expression.
As $m\rho$ is of length $m-1$ and $n\tau'$ is of length $n-m$, letting $k=n-m$ yields $n\tau'$ is of length $k$ and $m\rho$ is of length $n-k-1$, where $1 \leq k \leq n-3$ as $m$ ranges within $3 \leq m \leq n-1$, which gives us the sum from the desired recurrence.
\end{proof}

We now enumerate classes (E) and (F), demonstrating that (C), (E), and (F) are Wilf equivalent, namely being enumerated by the Catalan numbers.
We first introduce Catalan's triangle.
\begin{definition}[Entries of Catalan's triangle]
For $0 \leq k \leq n$, define $T(n,k)=\frac{n-k+1}{n+1}\binom{n+k}{n}$.
\end{definition}
This creates a triangle of entries, where rows correspond to values of $n$ and columns correspond to values of $k$.
It is well-known that the row-sums are the Catalan numbers.
% , but we provide a proof of this result using basic identities of binomial coefficients.
\begin{lemma}[\protect{\cite[Lemma 1]{bailey1996counting}}]\label{lemma: Catalan triangle row-sum Catalan}
For all $n \geq 0$, we have
\begin{equation*}
    \sum_{k=0}^n T(n,k) = C_{n+1}.
\end{equation*}
\end{lemma}
% \begin{proof}
% We wish to show that 
% \[ \sum_{k=0}^n \binom{n+k}{n}\frac{n-k+1}{n+1} = \frac{1}{n+2}\binom{2n+2}{n+1},\]
% or equivalently
% \[ \sum_{k=0}^n \binom{n+k}{n}(n-k+1) = \binom{2n+2}{n+2}. \]
% We use the well-known hockey-stick identity to do so:
% \begin{align*}
%     \sum_{k=0}^n \binom{n+k}{n}(n-k+1)
%     &= \sum_{k=0}^n \sum_{j=k}^n \binom{n+k}{n}
%     = \sum_{j=0}^n \sum_{k=0}^j \binom{n+k}{n}
%     = \sum_{j=0}^n \binom{n+j+1}{n+1}
%     \\ &= \binom{2n+2}{n+2},
% \end{align*}
% as desired.
% \end{proof}
It is also well-known that $T(n,k)$ satisfies the following recurrence.
% , but for completeness we provide a simple proof.
We will use the convention that $T(n,n+1)=0$, which is consistent with the original definition $T(n,k)=\binom{n+k}{n}\frac{n-k+1}{n+1}$ when $k=n+1$.
\begin{lemma}[\protect{\cite[Lemma 1]{bailey1996counting}}]\label{lemma: Catalan triangle recurrence}
For all $n \geq 1$ and $0 \leq k \leq n$, we have
\[ T(n,k) = \sum_{j=0}^k T(n-1,j).\]
\end{lemma}
% \begin{proof}
% We first re-express the right hand side as a double sum of binomial coefficients:
% \begin{align*}
%     \sum_{j=0}^k T(n-1,j)
%     &= \sum_{j=0}^k \binom{n-1+j}{n-1}\frac{n-j}{n}
%     \\ &= \frac{1}{n} \left((n-k-1)\sum_{j=0}^k \binom{n-1+j}{n-1} + \sum_{j=0}^k \binom{n-1+j}{n-1} (k+1-j)\right)
%     \\ &= \frac{1}{n} \left((n-k-1)\sum_{j=0}^k \binom{n-1+j}{n-1} + \sum_{j=0}^k \sum_{i=j}^k \binom{n-1+j}{n-1}\right).
%     % \\ &= \frac{1}{n} \left((n-k-1)\binom{n+k}{n} + \sum_{j=0}^k \sum_{i=j}^k \binom{n-1+j}{n-1}\right).
% \end{align*}
% Interchanging the order of summation and applying the hockey-stick identity yields
% \begin{align*}
%     \frac{1}{n} \left((n-k-1)\binom{n+k}{n} + \sum_{i=0}^k \binom{n+i}{n}\right)
%     &= \frac{1}{n}\left((n-k-1)\binom{n+k}{n} + \binom{n+k+1}{n+1}\right),
% \end{align*}
% which simplifies to
% \begin{align*}
%     \frac{1}{n}\binom{n+k}{n} \left(n-k-1+\frac{n+k+1}{n+1}\right)
%     &= \binom{n+k}{n} \frac{n-k+1}{n+1}
%     = T(n,k),
% \end{align*}
% as desired.
% \end{proof}

We first enumerate class (E).

\begin{theorem}\label{theorem: vincular 4 one vinculum E Catalan}
For all $n \geq 1$, we have $\left|\Av_n[\overline{14}23]\right|=C_{n-1}$.
\end{theorem}
\begin{proof}
For $n \leq 3$, we directly verify the result, where all cyclic permutations of length $n$ avoid $[\overline{14}23]$.

We now assume $n \geq 4$.
Let $i$ be a positive integer satisfying $1 \leq i \leq n-1$.
Let $A(n,i)$ denote the set of cyclic permutations $[\sigma]\in\Av_n[\overline{14}23]$ that have $i$ directly before $n$.
We will show that $|A(n,i)|=T(n-2,i-1)$.
Summing over $i$ would then yield
\[ \left|\Av_n[\overline{14}23]\right| = \sum_{i=1}^{n-1} T(n-2,i-1) = C_{n-1}, \]
using \cref{lemma: Catalan triangle row-sum Catalan}.
We prove this by induction on $n$, where it is trivial to verify this for the base case $n=3$.

For the inductive step, assume the result holds for $n-1$; using \cref{lemma: Catalan triangle recurrence}, it suffices to show that $A(n,i)$ is in bijection with $\bigcup_{j=1}^i A(n-1,j)$, which we denote $B(n,i)$.
Our bijection $\phi:A(n,i)\to B(n,i)$ will be to simply delete $n$ from a cyclic permutation of length~$n$.

We first show that this mapping takes cyclic permutations in $A(n,i)$ to cyclic permutations in $B(n,i)$.
As we have a cyclic ascent from $i$ to $n$, in order for $[\sigma]\in A(n,i)$ to avoid $[\overline{14}23]$, we must have the elements strictly between $i$ and $n$ appearing in decreasing order when reading from $n$, which implies that the Zeilberger set of $[\sigma]$ is equal to the interval $[i,n]$.
If the element preceding $n-1$ in $[\sigma]$ is $n$, then the element preceding $n-1$ in $\phi([\sigma])$ is $i$.
If the element preceding $n-1$ in $[\sigma]$ is not $n$, then this element is strictly less than $i$, for any element $k$ where $i \leq k < n$ occurs before $n$ when reading from $n-1$.
Therefore, the element preceding $n-1$ in $\phi([\sigma])$ is some $j$ for $1 \leq j \leq i$.
To show that $\phi([\sigma])$ is $[\overline{14}23]$-avoiding, as $[\sigma]$ is $[\overline{14}23]$-avoiding, it suffices to show that the newly created adjacency between $i$ and the element after it after removing $n$ cannot create a copy of $[\overline{14}23]$ where $i$ plays the role of 1.
In order for this to be a copy of $[\overline{14}23]$, we require $i$ to ascend to some element $k>i$, and as $[i,n-1]$ is arranged in descending order, this means $i$ may only ascend to $n-1$, but in this case we cannot find elements to play the roles of 2 and 3 for $[i,n-1]$ is in descending order.
Hence, $\phi([\sigma])$ is $[\overline{14}23]$-avoiding and $\phi([\sigma])\in B(n,i)$.

To prove $\phi$ is a bijection, we provide the inverse map $\psi$, which we claim is the operation that adds $n$ after $i$.
It suffices to show that this mapping sends elements of $B(n,i)$ to elements of $A(n,i)$, as once this is shown it clearly follows from the definition of $\psi$ that $\psi\circ\phi([\sigma])=[\sigma]$ for all $[\sigma]\in A(n,i)$ and $\phi\circ\psi([\tau])=[\tau]$ for all $[\tau]\in B(n,i)$.
Consider some cyclic permutation $[\tau]\in B(n,i)$.
Clearly, we have $i$ directly before $n$ in $\psi([\tau])$.
To show that $\psi([\tau])$ is still $[\overline{14}23]$-avoiding, as $n$ can only ever play the role of 4, it suffices to show that the cyclic ascent from $i$ to $n$ cannot be a part of a copy of $[\overline{14}23]$.
Equivalently, we must show that the elements in $[i,n]$ appear in decreasing order, when starting from $n$.
As $[\tau]$ is $[\overline{14}23]$-avoiding with a cyclic ascent from $j \leq i$ to $n-1$, we know $[j,n-1]$ appears in decreasing order, and thus $[i,n-1]$ appears in decreasing order in $[\tau]$ as well as $\psi[\tau]$.
As $n$ is inserted right after $i$ and thus between $i$ and $n-1$, we find the elements of $[i,n]$ appear in decreasing order.

Thus $\phi$ is a bijection, and the proof is complete.
\end{proof}

We now enumerate class (F).
The proof is similar in structure to that of \cref{theorem: vincular 4 one vinculum E Catalan}, but is slightly more involved and uses a different refinement than $A(n,i)$ from the previous proof.

\begin{theorem}\label{theorem: vincular 4 one vinculum F Catalan}
For all $n \geq 1$, we have $\left|\Av_n[\overline{14}32]\right|=C_{n-1}$.
\end{theorem}
\begin{proof}
For $n \leq 3$, we directly verify the result, where all cyclic permutations of length $n$ avoid $[\overline{14}32]$.

We now assume $n \geq 4$.
Let $i$ be a positive integer satisfying $0 \leq i \leq n-2$.
Let $E(n,i)$ denote the set of cyclic permutations $[\sigma]\in\Av_n[\overline{14}32]$ that have $\zeil[\sigma^r]=n-i$.
We will show by induction on $n$ that $|E(n,i)|=T(n-2,i)$.
As $2\leq\zeil[\sigma^r]\leq n$, summing over all $i$ satisfying $0 \leq i \leq n-2$ would then yield
\[ \left|\Av_n[\overline{14}32]\right| = \sum_{i=0}^{n-2} T(n-2,i) = C_{n-1}, \]
by \cref{lemma: Catalan triangle row-sum Catalan}.
It is trivial to verify the base case $n=3$.

For the inductive step, assume the result holds for $n-1$; using \cref{lemma: Catalan triangle recurrence}, it suffices to show that $E(n,i)$ is in bijection with $\bigcup_{j=0}^{i} E(n-1,j)$, which we denote $F(n,i)$.
Our bijection $\phi:E(n,i)\to F(n,i)$ will be to simply delete $n$ from $[\sigma]$, similar to the proof of \cref{theorem: vincular 4 one vinculum E Catalan}.

We first show that this mapping takes $[\sigma]\in E(n,i)$ to some $[\tau]\in F(n,i)$.
As $\zeil[\sigma^r]=n-i$, we find $i+1,i+2,\dots,n$ is a subsequence of $[\sigma]$, and thus $\phi([\sigma])$ has $i+1,\dots,n-1$ as a subsequence, so $n-1-i\leq\zeil[\phi([\sigma])^r]\leq n-1$.
Thus $\zeil[\phi([\sigma])^r]=n-1-j$ for some $0 \leq j \leq i$.
To show that $\phi([\sigma])\in\Av_{n-1}[\overline{14}32]$, suppose the element preceding $n$ is $a$ and the element following $n$ is $b$.
It suffices to show that the removal of $n$ cannot have caused a copy of $[\overline{14}32]$ with $a,b$ forming the vinculum $\overline{14}$.
As $[\sigma]$ is $[\overline{14}32]$-avoiding and $a,n$ form a cyclic ascent, we know that reading from $a$, the elements in $[a+1,n-1]$ are encountered in increasing order.
Hence in order for $a,b$ to be a cyclic ascent, $b$ must equal $a+1$, but then $a,b$ cannot form $\overline{14}$ as there are no elements with values between $a$ and $b=a+1$.
Thus, $\phi([\sigma])\in F(n,i)$.

To prove $\phi$ is a bijection, we provide the inverse map $\psi$.
For $[\tau]\in F(n,i)$ with $\zeil[\tau^r]=n-1-j$ for $0 \leq j \leq i$, if $j=i$ then define $\psi([\tau])$ to be the cyclic permutation obtained by inserting $n$ immediately after $n-1$, and if $j<i$ define $\psi([\tau])$ to be the cyclic permutation obtained by inserting $n$ immediately after $i$.
We first show that this mapping sends elements of $F(n,i)$ to elements of $E(n,i)$.
Notice that $[\tau]$ has $j+1,j+2,\dots,n-1$ as a subsequence, as $\zeil[\tau^r]=n-1-j$.

If $j=i$, then adding $n$ immediately after $n-1$ means $\zeil[\psi([\tau])^r]=n-i$, as desired.
To show that $\psi([\tau])$ is $[\overline{14}32]$-avoiding, it suffices to show that $n$ cannot be a part of a copy of $[\overline{14}32]$.
If it were, then $n$ must be the 4 in the $[\overline{14}32]$ pattern, but then $n-1$ would have to play the role of 1, leaving no possible elements to play the roles of 3 and 2.
Hence $\psi([\tau])\in E(n,i)$.

Otherwise if $j<i$, then we have that $j+1,\dots,i,n,i+1,\dots,n-1$ is a subsequence of $\psi([\tau])$, from which we can immediately see that the reverse Zeilberger set is the subsequence $i+1,\dots,n-1,n$, and thus $\zeil[\psi([\tau])^r]=n-i$.
To show that $\psi([\tau])$ is $[\overline{14}32]$-avoiding, it suffices to show that $n$ cannot be a part of a copy of $[\overline{14}32]$, where it would have to play the role of 4, and thus $i$ would have to play the role of 1.
Reading from $n$, we encounter the elements of $[i+1,n-1]$ in increasing order, and thus we cannot find two elements in $[i+1,n-1]$ to complete the $[\overline{14}32]$ pattern.
Thus $\psi([\tau])\in E(n,i)$.

It clearly follows from the definitions of $\phi$ and $\psi$ that $\phi\circ\psi([\tau])=[\tau]$ for all $[\tau]\in F(n,i)$.
To show $\psi\circ\phi([\sigma])=[\sigma]$ for all $[\sigma]\in E(n,i)$, notice that $[\sigma]$ has $i+1,i+2,\dots,n-1,n$ as a subsequence.

If $\phi([\sigma])\in E(n-1,i)$, then $\psi\circ\phi([\sigma])$ is obtained by inserting $n$ immediately after $n-1$ in $\phi([\sigma])$.
So, we must show that there does not exist an element between $n-1$ and $n$ in $[\sigma]$.
If $i=0$, then $[\sigma]=[\iota_n]$ and this holds.
Otherwise, element $i\geq 1$ is not located between $n-1$ and $i+1$ (reading in the forward direction), as $\zeil[\phi([\sigma])^r]=n-1-i$.
So $i$ is somewhere between $i+1$ and $n-1$.
If there was an element between $n-1$ and $n$, suppose the element immediately before $n$ was $x$.
Notice that $x<i$, so $x,n,i+1,i$ forms a $[\overline{14}32]$ pattern, contradicting the fact that $[\sigma]$ avoids $[\overline{14}32]$.
Hence, there is no element between $n-1$ and $n$ in $[\sigma]$, which yields $\psi\circ\phi([\sigma])=[\sigma]$.

Otherwise $\zeil[\phi([\sigma])^r]>n-1-i$, so $i$ is somewhere between $n-1$ and $i+1$.
As $\zeil[\sigma]=n-i$, however, $i$ cannot be between $n$ and $i+1$, so this means $i$ is somewhere between $n-1$ and $n$.
Hence $[\sigma]$ has $i+1,i+2,\dots,n-1,i,n$ as a subsequence.
In this case, $\psi\circ\phi([\sigma])$ is obtained by inserting $n$ immediately after $i$ in $\phi([\sigma])$, so we must show that there does not exist an element between $i$ and $n$ in $[\sigma]$.
If there were, suppose the element immediately before $n$ was $x$.
Note that $x<i$, so $x,n,i+1,i$ forms a $[\overline{14}32]$ pattern, contradicting the fact that $[\sigma]$ avoids $[\overline{14}32]$.
Therefore, there is no element between $i$ and $n$ in $[\sigma]$, which yields $\psi\circ\phi([\sigma])=[\sigma]$.

Hence, $\phi$ is a bijection, and the proof is complete.
\end{proof}

Lastly, we enumerate class (G) in terms of the number of strongly monotone partitions of $[n]$, which we now define.
\begin{definition}
A partition of $[n]$ is \emph{strongly monotone} if, when its parts are sorted so that their minimum elements are in increasing order, then their maximum elements are also in increasing order.
\end{definition}
Let $A_n$ denote the number of strongly monotone partitions of $[n]$.
Claesson and Mansour \cite{claesson2005enumerating} showed this sequence has the ordinary generating function
\[ \sum_{n \geq 0} A_n x^n = \frac{1}{1-x-x^2 B^*(x)},\]
where $B^*(x)$ denotes the ordinary generating function of the Bessel numbers, originally introduced by Flajolet and Schott \cite{flajolet1990non}.
While the size of the avoidance classes of (G) does not previously appear as a sequence in the OEIS \cite{sloane2018line}, it can be expressed in terms of $A_n$.

\begin{theorem}\label{theorem: vincular 4 one vinculum G strongly monotone}
For all $n \geq 2$, we have 
\[ \left|\Av_n[\overline{23}14]\right| = \sum_{i=0}^{n-2} \binom{n-2}{i} A_i.\]
\end{theorem}
\begin{proof}
We directly verify the result for $n \leq 3$, where $A_0=A_1=1$.

For general $n \geq 4$, consider $[\sigma]\in\Av_n[\overline{23}14]$.
Without loss of generality suppose that $\sigma_n = 1$, and say $n$ is at index $k$ for $1 \leq k \leq n-1$, i.e., $\sigma_k = n$.
The linear permutation $\sigma^{(n)}=\sigma_{k+1}\cdots\sigma_n\sigma_1\cdots\sigma_{k-1}$, i.e., the linear permutation obtained by rotating $\sigma$ so that $n$ is the last element and then removing $n$, must be $\overline{23}1$-avoiding.
We may partition $\sigma^{(n)}$ into blocks of consecutive elements, where each block is a maximal decreasing sequence, so that the transition from one block to the next is an ascent.
Claesson \cite[Proposition 5]{claesson2001generalized} characterized the $1\overline{32}$-avoiding permutations, where after partitioning a permutation into blocks of maximal increasing sequences, the permutation avoids $1\overline{32}$ if and only if the minima of the blocks are in decreasing order.
By reversing this characterization, we find the minimum elements of each block of $\sigma^{(n)}$ must be in increasing order.

Similarly, the linear permutation $\sigma^{(1)}=\sigma_1\cdots\sigma_{n-1}$ must be $3\overline{12}$ avoiding, as any copy of $3\overline{12}$ in $\sigma^{(1)}$ along with $\sigma_n=1$ would yield a copy of $14\overline{23}$, or a copy of $[\overline{23}14]$ in $[\sigma]$.
We partition $\sigma^{(1)}$ into blocks of maximal decreasing consecutive sequences.
By complementing Claesson's characterization \cite{claesson2001generalized} of $1\overline{32}$-avoiding permutations, we find the maximum elements of each block of $\sigma^{(1)}$ must be in increasing order.

Combining these two observations, we find that partitioning $\sigma$ into blocks of maximal decreasing consecutive sequences, the maximum elements of each block of $\sigma$ are in increasing order, the minimum elements of all but the last block of $\sigma$ are in increasing order, and the last block starts with $n$ and ends with 1.
See \cref{fig: _23_14 schematic} for a schematic diagram of this characterization of $[\sigma]$.

\begin{figure}[htbp!]
    \centering
    \begin{tikzpicture}[scale=1,thick]
        \draw[Circle-Circle] (0,1) -- (1,0);
        \draw[Circle-Circle] (1,2) -- (2,1);
        \draw[Circle-Circle] (2,3) -- (3,2);
        \draw[Circle-Circle] (3,4) -- (4,3);
        \draw[Circle-Circle] (4,5) -- (6,-1);
        \node at (4,5.3) {$n$};
        \node at (6,-1.3) {1};
    \end{tikzpicture}
    \caption{Schematic diagram of the plot of a $[\overline{23}14]$-avoiding cyclic permutation.}
    \label{fig: _23_14 schematic}
\end{figure}
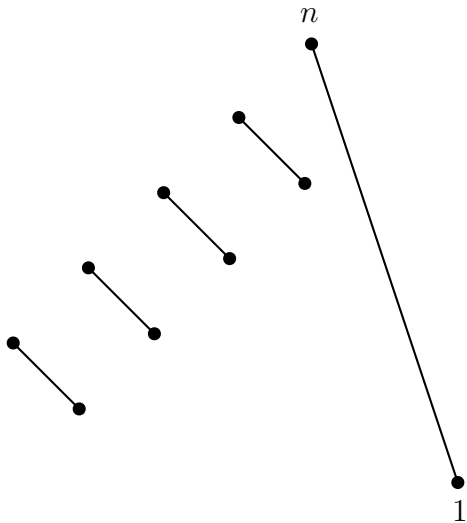

Our previous argument demonstrates that this characterization is a necessary condition for $[\sigma]$ to be $[\overline{23}14]$-avoiding.
We now show that this characterization is sufficient.
In order to have a copy of $[\overline{23}14]$ in such a cyclic permutation $[\sigma]$, we must have a cyclic ascent to play the role of $\overline{23}$, and in particular an ascent only occurs in between two blocks.
If the ascent is to the last block, it ascends to $n$, which cannot play the role of 3, but rather must play the role of 4; likewise, if the ascent is from the last block, it ascends from 1, which cannot play the role of 2, but rather must play the role of 1.
Otherwise the ascent is between two blocks, neither of which is the last block.
So our ascent is from the minimum element $a$ of block $j$ to the maximum element $b$ of block $j+1$, for some $j$.
In order to get an element smaller than $a$ to act as our 1, clearly we must go at least to the last block, the block from $n$ to 1.
But then it is impossible to get an element smaller than $b$ to act as our 4 before crossing $a$ again.

Thus, we now have a characterization of all $[\overline{23}14]$-avoiding permutations, as depicted in \cref{fig: _23_14 schematic}.
Suppose there are $i$ total elements in the blocks excluding the last block from $n$ to $1$, so that there are $n-i$ elements in the last block.
As the last block contains at least two elements, namely $n$ and 1, we find $0 \leq i \leq n-2$.
There are $\binom{n-2}{i}$ ways to pick the $i$ elements in these blocks, and the remaining $n-2-i$ elements other than $n$ and 1 must be in decreasing order between $n$ and 1.
Consider the sequence $\sigma' = \sigma_1\cdots\sigma_i$, i.e. the sequence consisting of all but the last block of $\sigma$.
The reduction of $\sigma'$ is a permutation on $[i]$, and it is easy to see that the partitioning of $\sigma'$ into its blocks, which satisfy the increasing minima and maxima conditions, provides a one-to-one correspondence to the strongly monotone partitions of $[i]$, where each block corresponds to one of the sets in a partitioning of $[i]$.
Hence there are $A_i$ ways to arrange the $i$ elements before the block from $n$ to 1.
Summing over all $0 \leq i \leq n-2$ thus yields
\[ \left|\Av_n[\overline{23}14]\right| = \sum_{i=0}^{n-2} \binom{n-2}{i} A_i,\]
as desired.
\end{proof}

We leave the enumeration of (H), which currently does not appear in the OEIS \cite{sloane2018line}, as an open problem.

\begin{remark}\label{remark: mansour shattuck 2341}
Since the writing of the original version of this paper, the case of avoiding $[\overline{23}41]$ has been dealt with by Mansour and Shattuck \cite{mansour2021enumerating}, where an explicit formula for the ordinary generating function enumerating the class has been found.
\end{remark}

\section{Minimal unavoidable sets of totally vincular patterns}\label{section: unavoidable sets of totally vincular patterns}
We call a vincular pattern of length $n$ with $n-1$ vincula \emph{totally vincular}.
A totally vincular pattern corresponds to requiring that all $n$ elements are adjacent, and thus when considering whether a permutation avoids a totally vincular pattern, we only consider consecutive blocks of $n$ elements.
Let $\overline{S_k}$ denote the set of totally vincular patterns of length $k$.
For a permutation $\pi \in S_k$, let the corresponding totally vincular pattern be denoted $\opi\in\overline{S_k}$.

In this section we consider which sets $\Pi\subseteq \overline{S_k}$ cannot be avoided by sufficiently long permutations, i.e., $\left|\Av_n[\Pi]\right|=0$ for all sufficiently large $n$.
We call such a set of patterns \emph{unavoidable}, and otherwise a set of patterns is \emph{avoidable}.
We say a set of patterns $\Pi$ is a \emph{minimal} unavoidable set if $\Pi$ is unavoidable, but any proper subset $\Pi'\subset \Pi$ is avoidable.

Very little is known about unavoidable pattern sets for ordinary permutations: see, for example, Wilf \cite[Section 10]{wilf2002patterns}.
On the other hand, the unavoidability of patterns is a well-studied and frequently-asked question for pattern avoidance of words, where letters come from a fixed finite alphabet but are allowed to be repeated; in the study of words, almost all pattern avoidance is implicitly totally vincular.

Recall that \cref{proposition: multiple vincular 3 doubleton 0} along with the enumeration in \cref{section: vincular 3} implies that the sets $\Pi$ consisting of the two totally vincular patterns of length 3 containing a 1 in the $i$th position for a fixed $i$, or the element-wise complements of these sets, i.e., the sets consisting of all possible patterns containing a 3 in the $i$th position, were minimal unavoidable sets.
We show that for general $k$, not just $k=3$, the same construction yields minimal unavoidable sets.

In particular, let $\Pi_{i,k}$ denote the set of all totally vincular patterns $\opi\in \overline{S_k}$ where $\opi_i = 1$ for $1 \leq i \leq k$.
This naturally means $\Pi_{i,k}^c$ is the set of all totally vincular patterns $\opi\in \overline{S_k}$ where $\opi_i = k$.
We then have the following theorem.

\begin{theorem}\label{theorem: vincular pi_i-k minimal unavoidable set}
For all positive integers $k$ and $i$ where $1 \leq i \leq k$, the sets $\Pi_{i,k}$ and $\Pi_{i,k}^c$ are minimal unavoidable sets.
\end{theorem}
\begin{proof}
We prove the result only for $\Pi_{i,k}$, as the $\Pi_{i,k}^c$ case then follows from trivial Wilf equivalence.
Likewise, by reversing $\Pi_{i,k}$ as necessary, we may assume $i \leq \frac{k+1}{2}$, as the other cases follow from trivial Wilf equivalence.

The case $k=1$ is trivial, as then $i=1$ and any element of a permutation of length $n \geq 1$ is order isomorphic to $1$, and $\Pi_{1,1}$ is a singleton set so a proper subset of $\Pi$ must be empty, which any permutation avoids.
The case $k=2$ is also trivial, as any cyclic permutation of length $n \geq 2$ must contain at least one copy of $[\overline{12}]$, which is the sole element of $\Pi_{1,2}$.
The case $k=3$ follows from \cref{proposition: multiple vincular 3 doubleton 0} and the enumerations of $\Av_n[\overline{123}]$ and $\Av_n[\overline{132}]$ from \cref{section: vincular 3}, where, in particular, these contain $[\delta_n]$ and $[\iota_n]$, respectively, showing that $\Pi_{i,k}$ is a minimal unavoidable set.

For general $k \geq 4$, we first show $\Pi_{i,k}$ is unavoidable for all $1 \leq i \leq \frac{k+1}{2}$, by showing no cyclic permutation of length at least $k$ avoids $\Pi_{i,k}$.
For any arbitrary $[\sigma]\in[S_n]$ for $n \geq k$, consider the consecutive subsequence of $k$ elements of $[\sigma]$ where the 1 is in position $i$ of the subsequence.
This consecutive subsequence must be order isomorphic to some element of $\Pi_{i,k}$, so no cyclic permutation of length $n \geq k$ can avoid $\Pi_{i,k}$.

To prove minimality of $\Pi_{i,k}$, we split into three cases: first, $i=1$; second, $1 < i < \frac{k+1}{2}$; and third, $i=\frac{k+1}{2}$ for odd $k$.

\textbf{Case 1.}
We first prove minimality for $i=1$.
It suffices to show that $\Pi_{1,k}\setminus\{\opi\}$ for some $\opi\in \overline{S_k}$ with $\opi_1 = 1$ is avoidable.
For $n \geq k$, consider $[\sigma]\in[S_n]$ given by
\[ \sigma = 1, n-k+\opi_2, n-k+\opi_3, \dots, n-k+\opi_k, n-k+1, n-k,\dots,2. \]
In other words, $\sigma$ consists of 1, followed by the permutation of $[n-k+2,n]$ that is order isomorphic to $\opi_2\cdots\opi_k$, following by the elements of $[2,n-k+1]$ in decreasing order.
See \cref{fig: minimal unavoidable case 1 schematic} for a schematic diagram of this construction of $\sigma$.
We claim that $[\sigma]$ avoids $\Pi_{1,k}\setminus\{\opi\}$.
We must show that for any cyclically consecutive sequence of $k$ elements of $\sigma$ for which the first element is the minimum element among these $k$ elements, this sequence is order isomorphic to $\opi$; otherwise, it would be order isomorphic to some other pattern in $\overline{S_k}$ that begins with 1, which by definition is an element of $\Pi_{1,k}\setminus\{\opi\}$.
As $n \geq k$, we find the $k$ consecutive elements starting at $\sigma_1=1$ are order isomorphic to $\opi$.
For any block of $k$ consecutive elements starting at $\sigma_j=n-k+\opi_j$ for $2 \leq j \leq k$, we have $\sigma_{k+1}=n-k+1<n-k+\opi_j$ is in this block of $k$ consecutive elements, so the first element $\sigma_j$ is not the minimum element among these $k$ consecutive elements.
For any block of $k$ consecutive elements starting at index $j$ for $k+1 \leq j \leq n$, i.e., starting in the portion decreasing from $n-k+1$ to 2, the second element is one smaller than the first element, so the first element is not the minimum element among these $k$ elements.
Hence, $[\sigma]$ avoids $\Pi_{1,k}\setminus\{\opi\}$, which is thus avoidable, implying $\Pi_{1,k}$ is a minimal unavoidable set of patterns.

\begin{figure}[htbp]
    \centering
    \begin{tikzpicture}[scale=0.8,thick]
        \filldraw (1,1) circle (2pt);
        \node at (1,0.5) {1};
        \filldraw (2,6) circle (2pt);
        % \node at (2,6.3) {$n-k+\opi_2$};
        \filldraw (3,7.5) circle (2pt);
        % \node at (3,7.8) {$n-k+\opi_3$};
        \filldraw (4,6.5) circle (2pt);
        % \node at (4,6.8) {$n-k+\opi_4$};
        \filldraw (5,7) circle (2pt);
        % \node at (5,7.3) {$n-k+\opi_5$};
        \draw[|-|] (1.8,8) -- (5.2,8);
        \node at (3.5,8.5) {$n+k+\opi_j$ for $2 \leq j \leq k$};
        \draw[Circle-Circle] (6,5.5) -- (10,1.5);
        \node at (7,5.8) {$n-k+1$};
        \node at (10,1) {2};
    \end{tikzpicture}
    \caption{Schematic diagram for the plot of $\sigma$ in Case 1 when $\opi = \overline{12534}$.}
    \label{fig: minimal unavoidable case 1 schematic}
\end{figure}
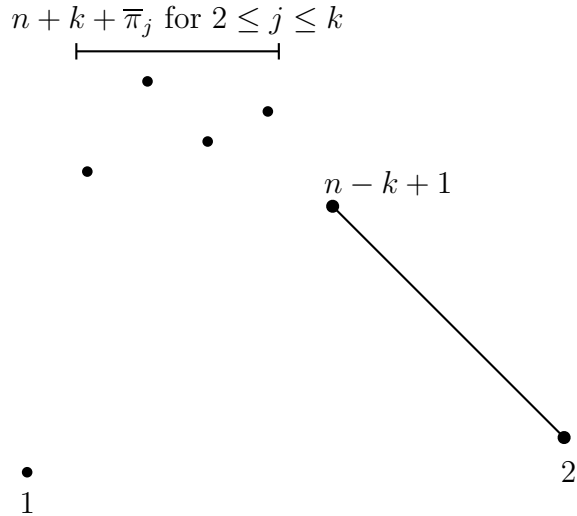

\textbf{Case 2.}
Next, we prove minimality for $1 < i < \frac{k+1}{2}$.
It suffices to show that $\Pi_{i,k}\setminus\{\opi\}$ for some $\opi\in\overline{S_k}$ with $\opi_i = 1$ is avoidable.
Express $\opi$ in the form $\opi=\overline{\rho1\tau}$ for $|\rho|<|\tau|$.
Consider the permutation $\sigma\in S_n$ where we take $\opi$ and map it to an order isomorphic permutation of $\{1\}\cup[n-k+2,n]$, and then append $[2,n-k+1]$ in decreasing order.
In other words,
\[ \sigma = n-k+\opi_1,\dots,n-k+\opi_{i-1},1,n-k+\opi_{i+1},\dots,n-k+\opi_k,n-k+1,n-k,\dots,2. \]
See \cref{fig: minimal unavoidable case 2 schematic} for a schematic diagram of this construction of $\sigma$.
We claim $[\sigma]$ avoids $\Pi_{i,k}\setminus\{\opi\}$.
Consider a consecutive subsequence of $k$ elements of $[\sigma]$.
If the $i$th element is 1, then the subsequence is order isomorphic to $\opi$, by construction.
It suffices to show if the $i$th element is not 1, then the $i$th element is not the minimum element in this consecutive subsequence, as then this subsequence cannot be order isomorphic to a pattern in $\Pi_{i,k}\setminus\{\opi\}$.
We now split into cases depending on the value of the $i$th element.
\begin{enumerate}
    \item If the $i$th element is in $[n-k+2,n]$, then it is either in the $\rho$ or $\tau$ portion of the permutation of $\{1\}\cup[n-k+2,n]$ order isomorphic to $\opi$.
    If it is in the $\tau$ portion, then this subsequence contains an element in $[2,n-k+1]$, which is smaller than the $i$th element.
    If it is in the $\rho$ portion, as $|\rho|<|\tau|$, the subsequence contains the next $|\tau|$ elements, which must include the 1, which is smaller than the $i$th element.
    \item If the $i$th element is in $[3,n-k+1]$, then the element of index $i+1$, which is part of the subsequence as $i<\frac{k+1}{2}\leq k$, is one less than it.
    \item If the $i$th element is 2, then the next $|\tau|>|\rho|$ elements are also in this subsequence, and namely this includes the 1, which is $|\rho|+1$ steps after the 2.
\end{enumerate}

\begin{figure}[htbp]
    \centering
    \begin{tikzpicture}[scale=0.8,thick]
        \filldraw (1,7.5) circle (2pt);
        \node at (2,0.5) {1};
        \filldraw (2,1) circle (2pt);
        \filldraw (3,6.5) circle (2pt);
        \filldraw (4,7) circle (2pt);
        \filldraw (5,6) circle (2pt);
        \draw[|-|] (0.8,8) -- (1.7,8);
        \draw[|-|] (2.3,8) -- (5.2,8);
        \draw[stealth-,dotted] (2,1.5) -- (2,8);
        \node at (3,8.5) {$n+k+\opi_j$ for $j\neq i$};
        \draw[Circle-Circle] (6,5.5) -- (10,1.5);
        \node at (7,5.8) {$n-k+1$};
        \node at (10,1) {2};
    \end{tikzpicture}
    \caption{Schematic diagram for the plot of $\sigma$ in Case 2 when $\opi = \overline{51342}$.}
    \label{fig: minimal unavoidable case 2 schematic}
\end{figure}

\textbf{Case 3.}
Third and finally, we prove minimality for $i=\frac{k+1}{2}$ when $k$ is odd.
It suffices to show that $\Pi_{i,k}\setminus\{\opi\}$ for some $\opi\in\overline{S_k}$ with $\opi_i = 1$ is avoidable.
Express $\opi$ in the form $\opi=\overline{\rho1\tau}$ for $|\rho|=|\tau|$.
We will prove the result when 2 is an element in $\rho$, as then by reversing, we prove the result for when 2 is an element in $\tau$.
Consider the permutation $\sigma$ where we take $\opi$ and map it to an order isomorphic permutation of $[1,2]\cup[n-k+3,n]$ and then append $[3,n-k+2]$ in decreasing order.
See \cref{fig: minimal unavoidable case 3 schematic} for a schematic diagram of this construction.
\begin{figure}[htbp]
    \centering
    \begin{tikzpicture}[scale=0.8,thick]
        \filldraw (1,7.5) circle (2pt);
        \filldraw (2,1.5) circle (2pt);
        \node at (2,1) {2};
        \filldraw (3,6.5) circle (2pt);
        \filldraw (4,1) circle (2pt);
        \node at (4,0.5) {1};
        \filldraw (5,5.5) circle (2pt);
        \filldraw (6,7) circle (2pt);
        \filldraw (7,6) circle (2pt);
        \draw[|-|] (0.8,8) -- (1.7,8);
        \draw[|-|] (2.3,8) -- (3.7,8);
        \draw[|-|] (4.3,8) -- (7.2,8);
        \draw[stealth-,dotted] (2,2) -- (2,8);
        \draw[stealth-,dotted] (4,1.5) -- (4,8);
        \node at (4,8.5) {$n+k+\opi_j$ for $\opi_j > 2$};
        \draw[Circle-Circle] (8,5) -- (11,2);
        \node at (9,5.3) {$n-k+2$};
        \node at (11,1.5) {3};
    \end{tikzpicture}
    \caption{Schematic diagram for the plot of $\sigma$ in Case 3 when $\opi = \overline{7251364}$.}
    \label{fig: minimal unavoidable case 3 schematic}
\end{figure}
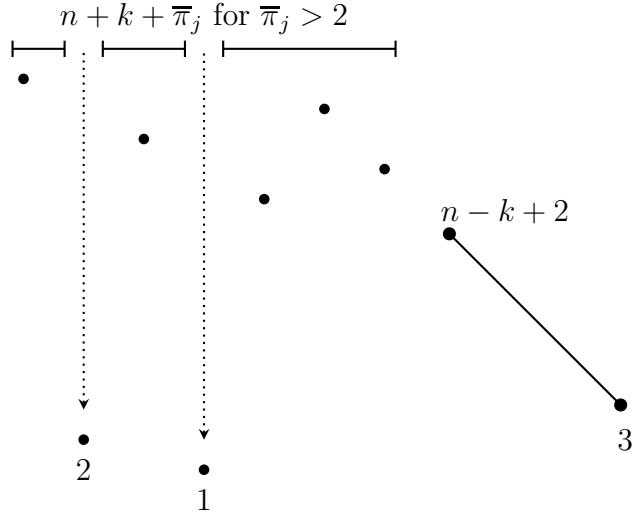

We claim $[\sigma]$ avoids $\Pi_{i,k}\setminus\{\opi\}$.
Consider a consecutive subsequence of $k$ elements of $[\sigma]$.
If the $i$th element is 1, then the subsequence is order isomorphic to $\opi$, by construction.
It suffices to show if the $i$th element is not 1, then the $i$th element is not the minimum element in this consecutive subsequence.
We now split into cases depending on the value of the $i$th element.
\begin{enumerate}
    \item If the $i$th element is in $[n-k+3,n]$, then it is either in the $\rho$ or $\tau$ portion of the permutation of $[1,2]\cup[n-k+3,n]$ order isomorphic to $\opi$.
    If it is in the $\tau$ portion, then this subsequence contains an element in $[3,n-k+2]$, which is smaller than the $i$th element.
    If it is in the $\rho$ portion, as $|\rho|=|\tau|$, the subsequence contains the next $|\tau|$ elements, which must include the 1, which is smaller than the $i$th element.
    \item If the $i$th element is in $[4,n-k+2]$, then the next element is one smaller than it, so it is not the minimum element in this subsequence.
    \item If the $i$th element is 3, then this block contains the 2 as the 2 is in the $|\rho|=|\tau|$ elements before the 1 and after the 3, so is not the minimum element.
    \item If the $i$th element is 2, as the 2 is in the $\rho$ portion of the permutation on the set $[1,2]\cup[n-k+3,n]$ order isomorphic to $\opi$, the subsequence contains the next $|\tau|=|\rho|$ elements, which must include the 1, so 2 is not the minimum.
\end{enumerate}

Hence, all $\Pi_{i,k}$ for $1 \leq i \leq k$ are minimal unavoidable sets.
\end{proof}

These sets $\Pi_{i,k}$ and $\Pi_{i,k}^c$ when $k=3$ correspond to the sets from \cref{proposition: multiple vincular 3 doubleton 0}.
It is easy to see from our analysis in \cref{section: vincular 3,section: multiple vincular 3} that these are the only minimal unavoidable subsets of $\overline{S_3}$.

For $k \geq 2$, viewing the power set of $\overline{S_k}$ as a Boolean lattice ordered by inclusion, it is clear that any chain can contain at most one minimal unavoidable set, as unavoidability is a monotone property.
Thus the number of minimal unavoidable sets is bounded above by the size of the maximum antichain of this lattice, which is $\binom{k!}{k!/2}$.
This yields the following proposition.
\begin{proposition}\label{proposition: upper bound minimal unavoidable sets}
For $k \geq 2$, the number of minimal unavoidable subsets of $\overline{S_k}$ is bounded above by $\binom{k!}{k!/2}$.
\end{proposition}

We conjecture that the sets from \cref{theorem: vincular pi_i-k minimal unavoidable set}, which have cardinality $(k-1)!$, are the smallest unavoidable subsets of $\overline{S_k}$.
Intuitively, this states that these sets are the ``most efficient" at preventing permutations from avoiding them.
\begin{conjecture}\label{conjecture: minimum unavoidable set}
For all $k \geq 1$, the minimum cardinality of an unavoidable subset of $\overline{S_k}$ is $(k-1)!$.
\end{conjecture}

\section{Maximum avoidable sets of totally vincular patterns}\label{section: max avoidable sets of totally vincular patterns}
The dual notion for a minimal unavoidable set $\Pi\subseteq\overline{S_k}$ is a \emph{maximal avoidable} set of patterns, an avoidable subset of $\overline{S_k}$ that, when any other vincular pattern $\opi \in \overline{S_k}$ is added, is no longer avoidable.

In this section, we determine the size of a \emph{maximum avoidable} set $\Pi\subseteq \overline{S_k}$, an avoidable subset of $\overline{S_k}$ whose cardinality is greater than or equal to any other avoidable subset of $\overline{S_k}$.
Clearly, all maximum avoidable sets are maximal avoidable sets.

\begin{theorem}\label{theorem: maximum avoidable set of totally vincular length k}
The maximum cardinality of an avoidable subset of $\overline{S_k}$ is $k!-k$.
\end{theorem}
\begin{proof}
We first show $k!-k$ is an upper bound on the cardinality of an avoidable set.
Suppose $\Pi\subseteq\overline{S_k}$ has cardinality $|\Pi|\geq k!-k+1$.
Partition $\overline{S_k}$ into $k$ sets depending on the index of the 1; in other words, using the notation from \cref{section: unavoidable sets of totally vincular patterns}, partition $\overline{S_k}$ into
\[ \overline{S_k} = \bigcup_{i=1}^k \Pi_{i,k}.\]
Each $\Pi_{i,k}$ has cardinality $(k-1)!$, and as $\Pi$ has cardinality $|\Pi| \geq k((k-1)!-1)+1$, we have $\Pi$ contains $\Pi_{i,k}$ for some $1 \leq i \leq k$.
This set is unavoidable by \cref{theorem: vincular pi_i-k minimal unavoidable set}.
Hence, any avoidable subset has cardinality at most $k!-k$.

We now show $k!-k$ can be achieved as the cardinality of an avoidable set.
Consider the set
\[ \Pi = \overline{S_k}\setminus\{\overline{12\cdots k},\overline{23\cdots k 1},\overline{34\cdots k 12},\dots,\overline{k12\cdots k-1}\}.\]
In other words, to avoid $\Pi$, every consecutive subsequence of $k$ elements must be order isomorphic to a rotation of $\iota_k = 12\cdots k$.
But observe $\iota_n=12\cdots n$ possesses this property, so it avoids $\Pi$.
Thus $\Pi$, which has cardinality $k!-k$, is a maximum avoidable set contained in $\overline{S_k}$.
\end{proof}

For $\pi \in S_k$, let $[\opi]$ denote the set of totally vincular patterns in $\overline{S_k}$ corresponding to the rotations of $\pi$, i.e., corresponding to $[\pi]$.
We have the following proposition on maximum avoidable subsets of $\overline{S_k}$.
\begin{proposition}\label{proposition: allow totally vincular cyclic permutation maximum avoidable subset}
For $\pi \in S_k$, the set $\Pi=\overline{S_k}\setminus[\opi]$ is a maximum avoidable set contained in $\overline{S_k}$.
\end{proposition}
\begin{proof}
As $|\Pi|=k!-k$, by \cref{theorem: maximum avoidable set of totally vincular length k} it suffices to show that $\Pi$ is avoidable, i.e., for arbitrarily large lengths $n$, we have $\left|\Av_n[\Pi]\right| > 0$.
We will show this holds when $n$ is a multiple of $k$.

Suppose $n=mk$ for a positive integer $m$.
Then consider the cyclic permutation $[\sigma]\in [S_n]$ given by
\begin{align*}
    \sigma = \, & m(\pi_1-1)+1, \dots, m(\pi_k-1)+1, m(\pi_1-1)+2, \dots, m(\pi_k-1)+2, \dots, \\
    & m(\pi_1-1)+m, \dots, m(\pi_k-1)+m.
\end{align*}
See \cref{fig: maximum avoidable blowup schematic} for a schematic diagram of this construction of $\sigma$.
\begin{figure}[htbp]
    \centering
    \begin{tikzpicture}[scale=0.5,thick]
        \filldraw (1,1) circle (4pt);
        \filldraw (2,9) circle (4pt);
        \filldraw (3,13) circle (4pt);
        \filldraw (4,5) circle (4pt);
        \draw[dotted] (5,1) -- (5,16);
        \filldraw (6,2) circle (4pt);
        \filldraw (7,10) circle (4pt);
        \filldraw (8,14) circle (4pt);
        \filldraw (9,6) circle (4pt);
        \draw[dotted] (10,1) -- (10,16);
        \filldraw (11,3) circle (4pt);
        \filldraw (12,11) circle (4pt);
        \filldraw (13,15) circle (4pt);
        \filldraw (14,7) circle (4pt);
        \draw[dotted] (15,1) -- (15,16);
        \filldraw (16,4) circle (4pt);
        \filldraw (17,12) circle (4pt);
        \filldraw (18,16) circle (4pt);
        \filldraw (19,8) circle (4pt);
    \end{tikzpicture}
    \caption{Schematic diagram for the plot of $\sigma$ when $\pi = 1342$ and $m=4$.}
    \label{fig: maximum avoidable blowup schematic}
\end{figure}
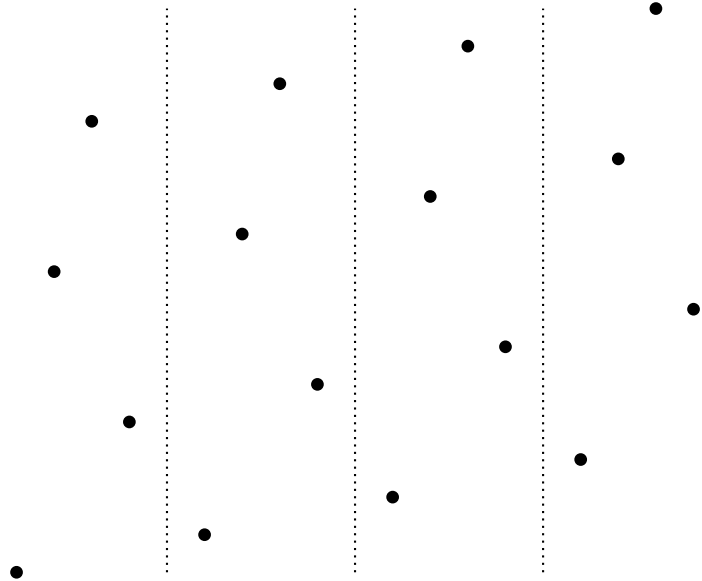

Essentially $\sigma$ consists of $\pi$ repeated $m$ times, vertically scaled with some slight shifts in order to not repeat any elements.
It is easy to see that any consecutive subsequence of $k$ elements in $[\sigma]$ is order isomorphic to some element of $[\opi]$, and thus $[\sigma]\in \Av_n[\Pi]$, so $\Pi$ is avoidable.
\end{proof}

We pose the following question.
\begin{question}\label{question: maximum avoidable sets}
Are there any other maximum avoidable subsets of $\overline{S_k}$, or are all maximum avoidable sets $\Pi\subset \overline{S_k}$ of the form $\Pi=\overline{S_k}\setminus[\opi]$ for some $\pi \in S_k$?
\end{question}

\section{Conclusion and open questions}\label{section: conclusion}
There are numerous avenues for further research regarding vincular pattern avoidance of cyclic permutations.
The enumeration corresponding to the last three Wilf equivalence classes of pairs of length 3 totally vincular cyclic patterns, listed in \cref{subsection: multiple vincular 3 doubleton}, is an open area for research.
One may also work on enumerating vincular cyclic patterns of length 4 with more than one vinculum, as well as sets of these length 4 patterns.
Non-vincular cyclic patterns of length 5 have not been enumerated yet, so this is also an open area of research, and a good first step before addressing vincular cyclic patterns of length 5.
Lastly, one could also consider enumerating avoidance classes of sets of patterns that do not all have the same length.

Regarding unavoidable sets, a characterization of unavoidable sets or even minimal unavoidable sets may be quite difficult, so a good first step would be to understand minimum unavoidable sets.
To this end, we leave \cref{conjecture: minimum unavoidable set} as an open problem to determine the minimum cardinality of an unavoidable set $\Pi\subseteq \overline{S_k}$.
Finding better bounds on the number of minimal unavoidable subsets of $\overline{S_k}$, i.e., improving \cref{proposition: upper bound minimal unavoidable sets}, could also be of interest.
Similarly, it may be quite difficult to characterize all the avoidable sets or even all the maximal avoidable sets, though empirical data suggest there are significantly fewer maximal avoidable sets than there are minimal unavoidable sets, and hence classifying maximal avoidable sets may be more tractable.
Characterizing the more restrictive class of maximum avoidable subsets is thus a good first step towards better understanding the avoidability of sets $\Pi\subseteq\overline{S_k}$, and we leave \cref{question: maximum avoidable sets} as an open area for research.

\section*{Acknowledgements}
We sincerely thank Amanda Burcroff and Benjamin Gunby for their input throughout the research process.
We also thank Daniel Zhu and Colin Defant for helpful ideas and input.
We would like to thank Prof.\ Joe Gallian for his editing feedback and operation of the Duluth REU program.
This research was conducted at the University of Minnesota Duluth Mathematics REU and was supported, in part, by NSF-DMS Grant 1949884 and NSA Grant H98230-20-1-0009.
Additional support was provided by the CYAN Mathematics Undergraduate Activities Fund.

\bibliographystyle{plain}
\bibliography{ref_vc}

\end{document}